\documentclass[11pt]{article}%
\usepackage{amssymb,amsmath,amsfonts,amsthm,array,bm,bbm,color}%
\usepackage{graphicx}
\providecommand{\U}[1]{\protect\rule{.1in}{.1in}}
\numberwithin{equation}{section}

\newtheorem{theorem}{Theorem}[section]
\newtheorem{lemma}[theorem]{Lemma}

\newtheorem{proposition}[theorem]{Proposition}
\newtheorem{remark}[theorem]{Remark}
\newtheorem{example}[theorem]{Example}
\newtheorem{definition}[theorem]{Definition}

\def\<{\langle}
\def\>{\rangle}
\def\E{\mathbb{E}}
\def\P{\mathbb{P}}
\def\R{\mathbb{R}}
\def\T{\mathbb{T}}
\def\Z{\mathbb{Z}}

\begin{document}
	
	\title{Finite time mixing and enhanced dissipation for 2D Navier-Stokes equations by Ornstein--Uhlenbeck flow}
	
	\author{Chang Liu\footnote{Email: liuchang2021@amss.ac.cn. School of Mathematical Sciences, University of Chinese Academy of Sciences, Beijing 100049, China and Academy of Mathematics and Systems Science, Chinese Academy of Sciences, Beijing 100190, China}
		\quad Dejun Luo\footnote{Email: luodj@amss.ac.cn. Key Laboratory of RCSDS, Academy of Mathematics and Systems Science, Chinese Academy of Sciences, Beijing 100190, China and School of Mathematical Sciences, University of Chinese Academy of Sciences, Beijing 100049, China}}
	\maketitle
	
	\vspace{-20pt}
	
	\begin{abstract}
		We consider the vorticity form of 2D Navier-Stokes equations perturbed by an Ornstein--Uhlenbeck flow of transport type. Contrary to previous works where the random perturbation was interpreted as Stratonovich transport noise, here we understand the equation in a pathwise manner and show the properties of mixing and enhanced dissipation for suitable choice of the flow.
	\end{abstract}
	

	\section{Introduction}
	Let $\T^2:=\R^2/ \Z^2$ be the two-dimensional (2D) torus; we consider the vorticity form of 2D Navier-Stokes equations on $\T^2$, perturbed by a smooth transport term:
	\begin{equation}\label{SPDE}
		\left\{
		\begin{aligned}
			&\partial_t \xi+u\cdot\nabla\xi+\bm{b}\cdot\nabla \xi=\kappa \Delta \xi, \\
			&u= K\ast \xi, \quad \xi|_{t=0}=\xi_0,
		\end{aligned}
		\right.
	\end{equation}
	where $\xi$ and $u$ are the fluid vorticity and velocity field, $K$ being the Biot-Savart kernel on $\T^2$:
	$$K\ast \xi:=-\nabla^\perp (-\Delta)^{-1}\xi,$$
	with $\nabla^\perp= (\partial_2, -\partial_1),\, \partial_i= \frac{\partial}{\partial x_i}$. $\kappa > 0$ is a fixed small number representing molecular diffusivity, and $\bm{b}:[0,\infty)\times \T^2\to \R^2$ is a random time-dependent and divergence free vector field, continuous in time and smooth in space variables, standing for the small-scale parts of fluid velocity, thus the term $\bm{b}\cdot\nabla \xi$ models the effects of fluid small scales on the larger component $\xi$. It is well known that for $L^2$-initial condition $\xi_0$, $\P$-a.s., equation \eqref{SPDE} admits a unique weak solution satisfying the usual energy estimate. Our purpose is to study, under suitable conditions on $\bm{b}$, the properties of mixing and dissipation enhancement for the system \eqref{SPDE}.
	
	The above model can be heuristically derived from the deterministic 2D Navier-Stokes equation by separating the fluid into large-scale components and small-scale ones, and modeling the corresponding small-scale velocity by a random field $\bm{b}$, see for instance \cite[Section 1.2]{FL21} for derivations in the 3D case and also \cite[Section 1.2]{Luo21} for similar discussions on 2D Boussinesq systems. In these papers, the small-scale perturbations are interpreted as Stratonovich multiplicative noises of transport type, delta-correlated in time and colored in space, and thus the first equation in \eqref{SPDE} is understood as a stochastic partial differential equation (SPDE); we refer to the recent works \cite{DebPapp, FP21, FP22, Holm} for more rigorous derivations of such stochastic fluid dynamics models. Indeed, the investigations of regularizing effects of multiplicative transport noise on various models began much earlier, see e.g. \cite{DFV14, FGP10, FGP11}. More recently, inspired by the scaling limit method introduced by Galeati \cite{Gal20}, stochastic fluid equations with transport noise have been studied intensively, and it is by now well understood that transport noise produces eddy dissipation/viscosity under certain rescaling of spatial covariance, see for instance \cite{CL23, FGL21a, FGL22, FlaLuo23, FLL23, FLuongo22, Luo21, LT23}. Moreover, the larger the noise intensity, the stronger the additional viscous term in the limit equations; the extra strong viscosity can be used to suppress possible blow-up of various deterministic equations, yielding long-term (even global) existence of strong solutions with large probability, cf. \cite{Agresti22, FGL21b, FHLN22, FL21, Luo23}. In a sense, the above results can be regarded as partial verifications of Boussinesq's eddy viscosity hypothesis \cite{Bous1877}, which is one of the basis for large eddy simulation (see \cite{Berselli}).
	
	However, noises with delta-correlation in time are just idealized approximations of real objects, and it is worthy of considering more practical perturbations, see \cite[Section 4]{MK99} for related discussions. As an attempt in this respect, Flandoli and Russo \cite{FR23} showed that the dissipation properties of a stochastic transport term of fractional Brownian motion with Hurst parameter $H>1/2$ are weaker than standard Brownian motion. In a slightly earlier work \cite{Pappalettera}, Pappalettera studied the mixing and dissipation enhancement properties of Ornstein-Uhlenbeck flows for passive scalar on $d$-dimensional torus $\T^d$:
	\begin{equation}\label{pappa-eq-1}
		\partial_t h+\bm{b}\cdot \nabla h= \kappa\Delta h, \quad h|_{t=0}= h_0\in L^2( \T^d).
	\end{equation}
	The time-dependent vector field $\bm{b}$ takes the form
	$$\bm{b}(t,x)= \sum_{j\in J} \bm{b}_j(x)\, \eta^{\alpha,j}(t),$$
	where $J$ is a finite index set, $\{\bm{b}_j\}_{j\in J}$ are divergence free vector fields on $\T^d$ and $\{\eta^{\alpha,j} \}_{j\in J}$ are independent real Ornstein-Uhlenbeck processes with covariance ${\rm Cov}(\eta^{\alpha,j}(t), \eta^{\alpha,j}(s))= \frac\alpha 2 \exp(-\alpha|t-s|)$, $\alpha>0$ being a parameter. As $\alpha\to \infty$, the covariance $\frac\alpha 2 \exp(-\alpha|t-s|)$ converges in distribution to the Dirac delta function, and thus $\eta^{\alpha,j} $ can be seen as approximations of the white noise. Assuming suitable conditions on the spatial properties of the family $\{\bm{b}_j\}_{j\in J}$, it is shown in \cite{Pappalettera} that the solution $h$ is close, in negative Sobolev norms, to the solution of the deterministic equation
	\begin{equation}\label{pappa-eq-2}
		\partial_t \bar h= (\kappa \Delta + \mathcal L) \bar h, \quad\bar h_0 =h_0,
	\end{equation}
	where the second order differential operator $\mathcal L \bar h= \sum_j \bm{b}_j\cdot \nabla (\bm{b}_j\cdot \nabla \bar h)$ stands for the enhanced dissipation. As mentioned above, the small-scale perturbations are understood in previous works (e.g. \cite{FGL21a, FGL22, Gal20}) as Stratonovich transport noise, and thus the additional operator $\mathcal L$ arises naturally as Stratonovich-It\^o corrector. Here, however, one has to compute the iterated integral of Ornstein-Uhlenbeck processes and to borrow some ideas from the proof of Wong-Zakai type results; see  \cite[Section 3]{FP22} for related computations. There are also many other advanced and very sophisticated works on mixing and dissipation enhancement properties, using different methods from ergodic theory, see e.g. \cite{BBPS21, BBPS22, GY21}.
	
	Motivated by \cite{Pappalettera}, we aim at studying the properties of mixing and enhanced dissipation of Ornstein-Uhlenbeck flow $\bm{b}$ for 2D Navier-Stokes equations \eqref{SPDE} in vorticity form. To overcome difficulties arising from the nonlinearity, we follow \cite{FGL21a, FGL22, Gal20} and assume that the time-dependent vector field $\bm{b}$ takes the more precise form
	$$ \bm{b}(t,x)=2\sqrt{\nu}\sum_{k\in \Z_0^2} \theta_k\, \sigma_k(x)\, \eta^{\alpha,k}(t),  $$
	where $\nu>0$ is the intensity of perturbation, $\Z_0^2=\Z^2 \backslash \{ 0 \}$ is the set of nonzero lattice points, and $\theta=\{\theta_k \}_{k\in \Z_0^2} \in \ell^2(\Z^2_0)$, the latter being the space of square summable sequences. We always assume that $\theta$ is a radial function of $k\in \Z_0^2$ with only finitely many nonzero components, and $\|\theta\|_{\ell^2}=(\sum_{k\in \Z_0^2} {\theta}_k^2) ^{1/2}=1$. The vector fields $\sigma_k(x) = \frac{k^\perp}{|k|} e^{2\pi i k\cdot x}$, where $k^\perp=(k_2,-k_1)$ and $k\cdot x= k_1x_1+ k_2x_2$, constitute a CONS of the space $L^2(\T^2,\R^2)$ of divergence free vector fields on $\T^2$ with zero mean, while $\eta^{\alpha,k}$ are independent Ornstein-Uhlenbeck processes as above.  Thanks to the exact choice of $\bm{b}$, the additional operator $\mathcal L$ takes the much simpler form $\nu \Delta $ (see \eqref{citation 1} below for related computations), and thus our purpose is to show that the solution $\xi$ of \eqref{SPDE} is close to that of the deterministic 2D Navier-Stokes equation with extra viscosity:
	\begin{equation}\label{PDE}
		\left\{ \aligned
		& \partial_t \bar\xi+ \bar u\cdot\nabla \bar\xi =(\kappa+\nu) \Delta \bar\xi, \\
		& \bar u= K\ast \bar\xi, \quad \bar\xi|_{t=0} = \xi_0.
		\endaligned
		\right.
	\end{equation}
Note that $\bar\xi$ has a fast exponential decay in $L^2$-norm (and also in negative Sobolev norms) for large $\nu$ which comes from the intensity of noise.
	
	To state more exactly our main results, we need some notation. For $s\in \R$, let $H^s=H^s(\T^2)$ be the usual Sobolev space on $\T^2$ endowed with the norm $\|\cdot \|_{H^s}$; we will write $H^0$ as $L^2$ and $\|\cdot \|_{H^0}$ as $\|\cdot \|_{L^2}$. Unless mentioned explicitly, we will use the same notation for spaces of functions and vector fields on $\T^2$. Since the equations \eqref{SPDE} and \eqref{PDE} preserve the spatial average of solutions, we shall assume that the spaces $H^s$ consist of functions of zero average.  We write $\|\theta\|_{\ell^{\infty}}$ for the supremum norm of $\theta\in \ell^2(\Z^2_0)$. In the sequel, the notation $a\lesssim b$ means that $a\le C b$ for some constant $C>0$; if we want to emphasize the dependence of $C$ on some parameters $\gamma,p$, then we write $a\lesssim_{\gamma,p} b$.

Here is the first main result of our work; it gives us a quantitative estimate on the distance, in terms of negative Sobolev norms, between the solutions $\xi$ and $\bar\xi$. Since $\bar\xi$ has a much faster decay in such norms, we can regard the result as a mixing property of the Ornstein-Uhlenbeck flow $\bm{b}$, valid on finite time intervals.
	
	\begin{theorem}\label{thm1}
		Let $\xi_0\in L^2(\T^2)$ and $\xi,\, \bar{\xi}$ be the unique solutions of \eqref{SPDE} and \eqref{PDE} respectively. Then for any $\gamma \in (0,\frac{1}{3})$, $\vartheta>0$ and $T \geq 1$, there exist $\zeta \in (0,1)$ and $\epsilon>0$ such that for $\alpha$ sufficiently large, it holds
		\begin{equation}\label{thm1.1}
			\E \big[\|\xi-\bar{\xi}\|_{C([0,T],H^{-\vartheta})}\big] \leq C_1  \|\xi_0\|_{L^2}  \exp\big(C_2 \|\xi_0\|^2_{L^2} \big) \big(\nu^{1+\frac{\gamma}{2}}\alpha^{-\epsilon}  +\nu^{\frac{1}{2}}\|\theta\|_{\ell^{\infty}}\big)^{\zeta} ,
		\end{equation}
		where $C_1>0$ is a constant depending on $\kappa, \nu, \zeta, \gamma, T$ and $C_2>0$ only depends on $\kappa,\nu, T$.
	\end{theorem}
	
We can make the right-hand of inequality \eqref{thm1.1} small by first choosing $\theta\in \ell^2(\Z^2_0)$ with small norm $\|\theta\|_{\ell^{\infty}}$ (see Examples \ref{Eb theta} and \ref{Eb theta N2N} in Section \ref{sec-prep} below), and then taking $\alpha$ big enough. This result is an analogue of \cite[Theorem 1.1]{Pappalettera}, where the author measured the closedness of solutions in the stronger H\"older space $C^\delta([0,1],H^{-\vartheta}),\, \delta>0$. The key idea in the proof of \cite[Theorem 1.1]{Pappalettera} is to express the difference $h-\bar h$ of solutions to \eqref{pappa-eq-1} and \eqref{pappa-eq-2} in terms of a random distribution $f$, see the beginning of \cite[Section 4]{Pappalettera} or \eqref{def of f} below for a similar quantity. If $f$ were differentiable in time, then $h-\bar h$ could be estimated using the mild expression involving $f$ and an analytic semigroup; in the absence of time regularity on $f$, one needs to apply \cite[Theorem 1]{Young integral} which can be thought of as a generalization of such estimates. In our case, we have to deal with the extra nonlinear terms in equations \eqref{SPDE} and \eqref{PDE}, thus we shall combine the above idea with the quantitative arguments developed in \cite{FLD quantitative}, and then apply the Gronwall lemma to get the desired estimate. Compared to \cite[Theorem 1.1]{FLD quantitative}, the coefficient $C_1$ in \eqref{thm1.1} might explode as $\kappa\to 0$, thus we cannot prove a similar estimate for the 2D Euler equation.
	
Our second main result shows the phenomenon of dissipation enhancement.
	
	\begin{theorem}\label{thm2}
		Given $\lambda>0$, $p \geq 1$ and $R>0$, we can find parameters $\nu>0$, $\alpha>0$ and $\theta \in \ell^2$, such that for every $\xi_0$ with $\|\xi_0\|_{L^2} \leq R$, there exists a random constant $C=C(\omega)>0$ with finite $p$-th moment, such that  the solution of \eqref{SPDE} satisfies the following exponential decay: $\P$-a.s.,
		$$\|\xi_t\|_{L^2} \leq C e^{-\lambda t} \|\xi_0\|_{L^2} \quad \mbox{for all } t\geq 0.$$
	\end{theorem}

This theorem improves \cite[Theorem 1.2]{Pappalettera} in two aspects: first, we deal with the nonlinear equation \eqref{SPDE} rather than the linear heat equation \eqref{pappa-eq-1}; second, the enhanced exponential decay of $\|\xi_t\|_{L^2}$ is shown for all sufficiently large $t>0$, instead of on a finite interval. We briefly discuss the key ingredients for proving the latter result. Note that the solution $\xi$ to \eqref{SPDE} is time homogeneous, due to the stationarity of the Ornstein-Uhlenbeck flow $\bm{b}$; combined with the estimate \eqref{thm1.1} restricted to the unit interval $[0,1]$, we conclude easily that similar result, up to taking conditional expectation, holds on any interval $[n,n+1]$ if equation \eqref{PDE} is restarted at time $t=n$ with the initial value $\xi_n$. As a consequence, we can show that $\E\|\xi_{n+1}\|_{L^2} \le c_0 \E\|\xi_{n}\|_{L^2}$ where $c_0>0$ can be very small by choosing parameters $\nu,\, \alpha$ and $\|\theta\|_{\ell^\infty}$ in a suitable way. Once we have such estimate, it is relatively standard to show the enhanced exponential decay; see Section \ref{subsec-proof-thm-2} for the detailed proofs. We mention that the initial condition $\xi_0$ is restricted in a ball of arbitrary (but fixed) radius $R$; this is due to the nonlinearity of \eqref{SPDE}, see the end of  \cite[Section 2.1]{Luo23b} for similar discussions.

	We make some further comments on the differences between our methods and those in \cite{Pappalettera}. First,  the main results of \cite{Pappalettera} are stated in dimension $d\ge 3$, while the corresponding 2D assertions are derived by assuming translation invariance in one direction, see the discussions in \cite[Remark 2.2]{Pappalettera}. The reason is due to a technical constraint on the Sobolev indices for product of functions: if $\phi\in H^a(\mathbb{T}^d)$ and $\psi\in H^b(\mathbb{T}^d)$ with $a, b<d/2$ and $a+b>0$, then one has $\phi\,\psi\in H^{a+b-d/2}$ and $\|\phi\,\psi\|_{H^{a+b-d/2}} \lesssim_{a,b,d} \|\phi \|_{H^a} \|\psi\|_{H^b}$, cf. \cite[Lemma 2.1]{Pappalettera} for the general case  $d\ge 2$, or Lemma \ref{HHH} below for the 2D case. If one wants to directly apply this result to estimate the $H^1$-norm of $\bm{b}(t)\cdot\nabla \phi$, where $\phi\in H^{2+\gamma}$ for some small $\gamma>0$, then a possible choice of parameters would be $a=d/2-\gamma<d/2,\, b=1+\gamma<d/2$, and one has
	$$\|\bm{b}(t)\cdot\nabla \phi \|_{H^1} \lesssim_\gamma \|\bm{b}(t) \|_{H^{d/2-\gamma}} \|\nabla \phi \|_{H^{1+\gamma}} \le \|\bm{b}(t) \|_{H^{d/2-\gamma}} \|\phi \|_{H^{2+\gamma}}; $$
	however, the above choice of parameters results in $d\ge 3,\, \gamma\in (0,(d-2)/2)$.
	In order to treat directly the 2D case, we make the following simple but key observation: since $\bm{b}(t)$ is divergence free in space, the function $\bm{b}(t)\cdot\nabla \phi$ has zero spatial average and thus one can apply Poincar\'e's inequality to get
	$$\|\bm{b}(t)\cdot\nabla \phi \|_{H^1} \lesssim \|\nabla(\bm{b}(t)\cdot\nabla \phi) \|_{L^2} \le \|\nabla\bm{b}(t)\cdot\nabla \phi \|_{L^2} + \|\bm{b}(t)\cdot \nabla^2 \phi \|_{L^2}; $$
	note that now we only need to estimate $L^2$-norm of products, it is possible to choose suitable parameters such that the above product rule of Sobolev functions is applicable in the 2D case, see \eqref{L2 decompose} below for details.
	
	Next, we have tried to avoid using the supremum in time of Sobolev norms of $\bm{b}(t,\cdot)$, with one exception in Lemma \ref{I25}; in this way, most of the estimates do not involve logarithmic terms, making them look simpler than those in \cite{Pappalettera}.
	
	We finish the short introduction with the structure of the paper. We present some preliminary results in Section 2 which will be frequently used below. Then we prove in Section 3 a few useful estimates on the solution $\xi$ to equation \eqref{SPDE}; as in \cite{Pappalettera}, the main technical estimate is Proposition \ref{main proposition} whose proof will be postponed to Section 5 in order not to interrupt the readability of the paper. The main results (Theorems \ref{thm1} and \ref{thm2}) will be proved in Section 4, again following some ideas in \cite{Pappalettera} with suitable modifications to deal with the nonlinearities.

	\section{Preparations} \label{sec-prep}
	
	Recall that $\Z_0^2$ consists of 2D nonzero integer points; let $\{W^k\}_{k\in \Z^2_0}$ be a family of independent two-sided Brownian motions defined on some filtered probability space $(\Omega,\{ \mathcal{F}_t\}_{t\in \R}, \P)$. For every $\alpha>1$, the processes
	$$	\eta^{\alpha,k}(t):=\int_{-\infty}^{t}\alpha e^{-\alpha(t-s)}\, dW_s^k , \quad t\geq 0,\  k\in \Z^2_0 $$
	constitute a family of independent Ornstein-Uhlenbeck processes, which are solutions of the 1D SDE
	$$ d \eta^{\alpha,k} =-\alpha\, \eta^{\alpha,k}\, dt+\alpha\, dW_t^k . $$
	It is clear that $\eta^{\alpha,k}$ is a stationary process, with the invariant Gaussian measure $N(0, \alpha/2)$. For the reader's convenience, we recall that the random vector field $\bm{b}$ is defined as
	$$ \bm{b}(t,x)=2\sqrt{\nu}\sum_{k\in \Z_0^2} \theta_k\, \sigma_k(x)\, \eta^{\alpha,k}(t),$$
	where $\nu>0$, $\theta\in \ell^2(\Z^2_0)$ is radially symmetric and has compact support, $\|\theta \|_{\ell^2}=1$, and $\sigma_k(x)=\frac{k^\perp}{|k|} e^{2\pi ik\cdot x},\, k\in \Z_0^2 $ constitute a CONS of the space of divergence free vector fields in $L^2(\T^2,\R^2)$.
	
	We next introduce the definition of weak solutions for the equation \eqref{SPDE}, namely
	$$\partial_t \xi+u\cdot\nabla\xi+\bm{b}\cdot\nabla \xi=\kappa \Delta \xi, $$
	with $u=K\ast \xi$ and initial data $ \xi_0 \in L^2(\mathbb{T}^2)$.
	
	\begin{definition}\label{solution of 1.1}
		Suppose $\xi_0 \in L^2(\mathbb{T}^2)$. A stochastic process $\xi: \Omega \times [0,\infty) \rightarrow L^2(\T^2)$ is called a weak solution of \eqref{SPDE}, if there exists a $\P$-negligible set $\mathcal{N}\subset \Omega$ such that for every $\omega \in \mathcal{N}^c$, it holds $\xi(\omega,\cdot)\in L^{\infty}\big([0,\infty),L^2(\T^2)\big)$ and
		$$ \langle \phi, \xi_t \rangle=\<\phi, \xi_s\>+ \int_{s}^{t}\< u_r\cdot\nabla\phi, \xi_r\> \, dr +\int_{s}^{t} \< \bm{b}(r)\cdot\nabla\phi, \xi_r\> \, dr+\kappa\int_{s}^{t}\<\Delta\phi, \xi_r\> \, dr, $$
		for every test function $\phi\in C^{\infty}(\T^2)$ and every $0\leq s<t<\infty $.
	\end{definition}
	
	It is easy to know that, given any $L^2$-initial condition $\xi_0$, \eqref{SPDE} admits a unique weak solution satisfying the following $\P$-a.s. energy estimate:
	\begin{equation} \label{priori estimates 1}
		\sup_{t\in [0,\infty)}\Big(\|\xi_t\|^2_{L^2}+2\kappa \int_{0}^{t}\|\nabla \xi_s\|^2_{L^2} \, ds \Big) \leq \|\xi_0\|^2_{L^2}.
	\end{equation}
	Similarly, for the solution of \eqref{PDE}, it holds
	\begin{equation} \label{priori estimates 2}
		\sup_{t\in [0,\infty)}\Big(\|\bar{\xi}_t\|^2_{L^2}+2(\kappa+\nu) \int_{0}^{t}\|\nabla \bar{\xi}_s\|^2_{L^2} \, ds \Big) \leq \|\xi_0\|^2_{L^2}.
	\end{equation}	
	
	Now we state several technical lemmas for later use; as they are classical results in harmonic analysis, we omit their proofs. The first result is concerned with the product of Sobolev functions, see e.g. \cite[Corollary 2.55]{BCD11}.
	\begin{lemma}\label{HHH}
		For any $s_1,s_2 <1$, if $s_1+s_2>0$, then for any $u \in H^{s_1}(\T^2)$ and $v \in H^{s_2}(\T^2)$, we have $uv \in H^{s_1+s_2-1}(\T^2)$, and the following inequality holds:
		$$	\|uv\|_{H^{s_1+s_2-1}} \lesssim \|u\|_{H^{s_1}} \|v\|_{H^{s_2}}. $$
	\end{lemma}
	
The following result follows easily from Gagliardo-Nirenberg's characterization of $H^{\alpha}$-norm for $\alpha\in (0,1)$, cf. \cite[Proposition 1.37]{BCD11}

	\begin{lemma}\label{HCH}
		Let $\alpha\in (0,1)$ and $\epsilon>0$ be such that $\alpha+\epsilon<1$, then for any $u \in C^{\alpha+\epsilon}(\T^2)$ and $v \in H^{\alpha}(\T^2)$,  it holds
		$$	\|uv\|_{H^{\alpha}} \lesssim \|u\|_{C^{\alpha+\epsilon}} \|v\|_{H^{\alpha}}. $$
	\end{lemma}

	\begin{lemma}[Interpolation inequality] \label{interpolation}
		For any $s_1<s<s_2$, there exists $\alpha \in (0,1)$ satisfying $s=\alpha s_1+(1-\alpha)s_2$, such that
		$$	\|u\|_{H^{s}} \leq \|u\|^{\alpha}_{H^{s_1}} \|u\|^{1-\alpha}_{H^{s_2}}.	$$
	\end{lemma}
	
	The next lemma gives a useful estimate on the Sobolev norms of the vector field $\bm{b}$.
	
	\begin{lemma}\label{lemma Eb}
		For every $p\geq 2$ and $\tau>0$, we have the following estimate:
		$$	\sup_{s\ge 0} \E\big[\|\bm{b}(s)\|^p_{H^{\tau}} \big] \lesssim \nu^{\frac{p}{2}}  \alpha^{\frac{p}{2}} \, C_{\theta,\tau,p} \, , $$
		where $C_{\theta,\tau,p}:=\sum_{k \in \mathbb{Z}_0^2} \theta_{k}^2 \, |k|^{p\tau} \in (0,\infty) $ is a constant depending on $\theta,\tau,p$.
	\end{lemma}
	
	\begin{proof}
		Recall that $\sum_{k\in \Z_0^2} \theta_{k}^2=1$, then by the definition of $\bm{b}(s)$, Jensen's inequality yields
		\begin{equation}\label{b Jensen}
			\|\bm{b}(s)\|^p_{H^{\tau}}=\Big(4\nu \sum_{k \in \mathbb{Z}_0^2} \theta_{k}^2 \, (\eta^{\alpha,k}(s))^2 \, |k|^{2\tau}\Big)^{\frac{p}{2}} \leq (4\nu)^{\frac{p}{2}} \sum_{k \in \mathbb{Z}_0^2} \theta_{k}^2 \, \big|\eta^{\alpha,k}(s)\big|^p \, |k|^{p\tau}.
		\end{equation}
		Taking expectation, we arrive at
		\begin{equation*}
			\mathbb{E}\big[\|\bm{b}(s)\|^p_{H^{\tau}} \big] \leq (4\nu)^{\frac{p}{2}} \sum_{k \in \mathbb{Z}_0^2} \theta_{k}^2 \, |k|^{p\tau} \  \mathbb{E}  \Big[\big|\eta^{\alpha,k}(s)\big|^p \Big]\lesssim (\nu\alpha)^{\frac{p}{2}} \sum_{k \in \mathbb{Z}_0^2} \theta_{k}^2 \, |k|^{p\tau}
		\end{equation*}
		which gives us the desired estimate.
	\end{proof}
	
	\begin{remark}\label{Jensen C}
		For $n>1$, by Jensen's inequality we have $C_{\theta,\tau,p/n} \leq C^{1/n}_{\theta,\tau,p}$.
	\end{remark}
	
	In the following two examples, we compute explicitly the values of $\|\theta\|_{\ell^{\infty}}$ and $C_{\theta,\tau,p}$ for special choices of coefficients $\theta$.
	
	\begin{example}\label{Eb theta}
		For $N\ge 1$, we define $\theta\in \ell^2(\Z^2_0)$ as follows:
		\begin{equation}\nonumber
			\theta_k = \varepsilon_N \frac{1}{|k|^{a}} \mathbf{1}_{\left\{1 \leq |k| \leq N\right\}}, \quad k\in \Z_0^2,
		\end{equation}
		where $a\in (0,1)$, $\varepsilon_N$ is  a normalizing constant depending on $N$ such that $\|\theta\|_{\ell^2}=1$. Then for every $p \geq 1$, it holds
		\begin{equation*}
			\aligned
			\|\theta\|_{\ell^{\infty}} &\sim \Big(\frac{1-a}{\pi}\Big)^{1/2} \big(N^{2-2a}-1\big)^{-1/2}  \to 0 \quad \mbox{as } N\to \infty ,\\
			C_{\theta,\tau,p} &\sim \frac{2-2a}{2-2a+p\tau} \frac{N^{2-2a+p\tau}-1}{N^{2-2a}-1} \sim N^{p\tau} \quad \mbox{as } N\to \infty .
			\endaligned
		\end{equation*}
	\end{example}
	
	\begin{proof}
		By the definition of $C_{\theta,\tau,p}$, we can get
		$$C_{\theta,\tau,p}=\varepsilon_N^2 \, \sum_{1\leq |k|\leq N}  |k|^{-2a+p\tau}=:\varepsilon_N^2 \, d_N \, ,$$
		where $d_N=\sum_{1\leq |k|\leq N} |k|^{-2a+p\tau} $. Notice that $\|\theta\|^2_{\ell^2}=\sum_{1\leq |k|\leq N} \varepsilon^2_N |k|^{-2a}=1$, then we can estimate $\varepsilon_N^2=\big(\sum_{1\leq |k|\leq N}  |k|^{-2a}\big)^{-1} $ by integration as follows:
		\begin{equation}\nonumber
			\varepsilon_N^2 \sim \Big(\int_{1\leq |x| \leq N} |x|^{-2a}\, dx \Big)^{-1} =\Big(\int_{1}^{N} \int_{0}^{2\pi} \frac{r}{r^{2a}}\, d\varphi\, dr  \Big)^{-1} =\frac{1-a}{\pi} \big(N^{2-2a}-1\big)^{-1}.
		\end{equation}
		Furthermore, we can easily get $\|\theta\|_{\ell^{\infty}} =\varepsilon_N$ for a fixed $N$. In the same way, we can estimate $d_N\sim \frac{2\pi}{2-2a+p\tau}\big(N^{2-2a+p\tau}-1\big)$, and therefore we obtain the value of $C_{\theta,\tau,p}$.
	\end{proof}
	
	\begin{example}\label{Eb theta N2N}
		If we change the support set of the above example and define
		\begin{equation}\nonumber
			\theta_k = \varepsilon_N \frac{1}{|k|^{a}} \mathbf{1}_{\left\{N \leq |k| \leq 2N\right\}}, \quad k\in \Z_0^2,
		\end{equation}
		where $a> 0$ and $\varepsilon_N$ is  still a normalizing constant, then for every $p\geq 1$, we can obtain
		$$ \aligned
		\|\theta\|_{\ell^{\infty}} &\sim \begin{cases}
			\big(\frac{1-a}{\pi}\big)^{1/2} \big(2^{2-2a}-1\big)^{-1/2} N^{-1}, & 0<a< 1, \\
			\big(2\pi \log2\big)^{-1/2} N^{-1}, & a=1, \\
			\big(\frac{a-1}{\pi}\big)^{1/2} \big(1-2^{2-2a}\big)^{-1/2} N^{-1} ,& a>1;
		\end{cases} \\
		C_{\theta,\tau,p} &\sim \begin{cases}
			\frac{(2-2a)(2^{2-2a+p\tau}-1)}{(2-2a+p\tau)(2^{2-2a}-1)} \, N^{p\tau} ,& a \neq 1, \\
			\frac{2^{p\tau}-1}{ p\tau \log2 } \, N^{p\tau}, & a=1.
		\end{cases}
		\endaligned $$
	\end{example}	
	
	Finally we present a moment estimate of $\bm{b}$ in the space $C([0,T], H^{\tau})$, $T\geq 1$, which will be used in Section 5.4.
	\begin{lemma}\label{Eb sup lemma}
		Consider $\bm{b}$ as defined before, then for every $p\geq 2$ and $\tau>0$, it holds
		\begin{equation}\nonumber
			\mathbb{E}\Big[\sup_{s\in [0,T]}\|\bm{b}(s)\|^p_{H^{\tau}}\Big] \lesssim  \nu^{\frac{p}{2}} \alpha^{\frac{p}{2}} \, C_{\theta,\tau,p} \log^{\frac{p}{2}}(1+\alpha T), \quad \forall\, T \geq 1.
		\end{equation}	
	\end{lemma}
	\begin{proof}
		Recall the useful estimate from \cite{ZGH OU}: for every fixed $p\geq 1$, it holds
		\begin{equation}\nonumber
			\mathbb{E}\Big[\sup_{s\in [0,T]}|\eta^{\alpha,k}(s)|^p\Big] \lesssim \alpha^{\frac{p}{2}} \log^{\frac{p}{2}}(1+\alpha T)  \quad \mbox{for all } k \in\mathbb{Z}_0^2;
		\end{equation}
		by \eqref{b Jensen}, for $p \geq 2$, we take supremum and then expectation on $	\|\bm{b}(s)\|^p_{H^{\tau}}$ to obtain
		\begin{equation}\nonumber
			\mathbb{E}\Big[\sup_{s\in [0,T]}\|\bm{b}(s)\|^p_{H^{\tau}}\Big] \le (4\nu)^{\frac{p}{2}} \sum_{k\in \Z_0^2} |k|^{p\tau} \theta_{k}^2 \, \mathbb{E}\Big[\sup_{s\in [0,T]}|\eta^{\alpha,k}(s)|^p\Big] \lesssim   \nu^{\frac{p}{2}}  \alpha^{\frac{p}{2}} \, C_{\theta,\tau,p}  \log^{\frac{p}{2}}(1+\alpha T).
		\end{equation}
	\end{proof}

	\section{Useful Estimates}
	
	We first prove an estimate on the time increment of $\xi$ in $H^{-1}$-norm, which will be repeatedly used in the proofs of Lemma \ref{E H2-g} and Proposition \ref{main proposition}. Thanks to the estimate \eqref{priori estimates 1}, we often control $\|\xi_s\|_{L^2}$ by $\|\xi_0\|_{L^2}$ in the following proofs.

	\begin{lemma}\label{E H^-1}
		Let $t \geq0$, $\delta \in (0,1)$ satisfy $\delta \alpha \gtrsim 1$, then for every $p \geq 2$ and $\gamma \in (0,1)$, it holds
		$$ 	\mathbb{E}\Big[\|\xi_{t+\delta}-\xi_t\|^p_{H^{-1}}\Big] \lesssim \delta^p \nu^{\frac{p}{2}} \alpha^{\frac{p}{2}}   \, C_{\theta,1+\gamma,p} \,  \|\xi_0\|^p_{L^2}\big(1+\|\xi_0\|^p_{L^2}\big) . $$
	\end{lemma}

	\begin{proof}
		By Definition \ref{solution of 1.1}, for every test function $\phi\in C^{\infty}(\mathbb{T}^2)$, we have
		$$	\big|\langle \phi, \xi_{t+\delta}-\xi_t \rangle\big| \leq \int_{t}^{t+\delta}\big|\langle u_s\cdot \nabla \phi, \xi_s \rangle\big| \, ds + \int_{t}^{t+\delta}\big|\langle \bm{b}(s)\cdot \nabla \phi, \xi_s \rangle\big| \, ds +\kappa\int_{t}^{t+\delta}\big|\langle \Delta\phi, \xi_s \rangle\big| \,ds. $$
		
		Now we will deal with each term separately. First, according to Sobolev embedding theorem and Lemma \ref{interpolation}, for $\gamma \in (0,1)$, we have the following estimate:
		\begin{equation}\nonumber
			\big|\langle u_s\cdot \nabla \phi, \xi_s \rangle\big| \, \leq \,  \|u_s\|_{L^{\infty}} \|\nabla \phi\|_{L^{2}} \|\xi_s\|_{L^{2}}
			\, \lesssim \, \|\xi_s\|_{H^{\gamma}} \|\phi\|_{H^{1}} \|\xi_0\|_{L^{2}}
			\, \lesssim \, \|\xi_s\|^{\gamma}_{H^1} \|\phi\|_{H^{1}} \|\xi_0\|^{2-\gamma}_{L^{2}}.
		\end{equation}
		Hence we can use H\"older's inequality and \eqref{priori estimates 1} to get
		\begin{equation}\label{Lemma 3.1 1st term}
			\begin{split}
				\int_{t}^{t+\delta}\big|\langle u_s\cdot \nabla \phi, \xi_s \rangle\big| \, ds & \lesssim  \|\phi\|_{H^{1}} \|\xi_0\|^{2-\gamma}_{L^{2}} \int_{t}^{t+\delta} \|\xi_s\|^{\gamma}_{H^1} \, ds\\
				&\leq \|\phi\|_{H^{1}} \|\xi_0\|^{2-\gamma}_{L^{2}}  \Big(\int_{t}^{t+\delta} \|\xi_s\|^{2}_{H^1} \, ds \Big)^{\frac{\gamma}{2}} \Big(\int_{t}^{t+\delta} 1 \, ds \Big)^{1-\frac{\gamma}{2}}\\
				&\leq \kappa^{-\frac{\gamma}{2}} \delta^{1-\frac{\gamma}{2}} \|\phi\|_{H^{1}} \|\xi_0\|^{2}_{L^{2}} 	.
			\end{split}
		\end{equation}
		In the same way, we can estimate the second term. For $\gamma\in (0,1)$, we have
		\begin{equation}\nonumber
			\begin{split}
				\big|\langle \bm{b}(s)\cdot \nabla \phi, \xi_s \rangle\big| \, \leq \, \|\bm{b}(s)\|_{L^{\infty}} \|\nabla \phi\|_{L^{2}} \|\xi_s\|_{L^{2}} \, \lesssim \, \|\bm{b}(s)\|_{H^{1+\gamma}} \|\phi\|_{H^{1}} \|\xi_0\|_{L^{2}},
			\end{split}
		\end{equation}
		then the following inequality holds:
		\begin{equation}\label{Lemma 3.1 2nd term}
			\int_{t}^{t+\delta}\big|\langle \bm{b}(s)\cdot \nabla \phi, \xi_s \rangle\big| \, ds \lesssim \|\phi\|_{H^{1}} \|\xi_0\|_{L^{2}} \int_{t}^{t+\delta} \|\bm{b}(s)\|_{H^{1+\gamma}} \, ds.
		\end{equation}
		As for the last term, \eqref{priori estimates 1} and H\"older's inequality yield
		\begin{equation}\label{Lemma 3.1 3rd term}
			\begin{split}
				\kappa\int_{t}^{t+\delta}\big|\langle \Delta\phi, \xi_s \rangle\big| \, ds & \leq \kappa\int_{t}^{t+\delta} \| \Delta\phi\|_{H^{-1}}  \|\xi_s\|_{H^{1}} \,ds\\
				&\leq \kappa \, \|\phi\|_{H^{1}} \Big( \int_{t}^{t+\delta} \|\xi_s\|^2_{H^{1}} \, ds \Big) ^{\frac{1}{2}} \Big(\int_{t}^{t+\delta} 1 \, ds \Big)^{\frac{1}{2}}\\
				&\leq  \kappa^{\frac{1}{2}} \delta^{\frac{1}{2}} \|\phi\|_{H^{1}} \|\xi_0\|_{L^{2}} .
			\end{split}
		\end{equation}	
		
		Having \eqref{Lemma 3.1 1st term}-\eqref{Lemma 3.1 3rd term} at hand, we deduce
		$$	\big|\langle \phi, \xi_{t+\delta}-\xi_t \rangle\big| \lesssim \|\phi\|_{H^{1}} \bigg[ \kappa^{-\frac{\gamma}{2}} \delta^{1-\frac{\gamma}{2}}  \|\xi_0\|^{2}_{L^{2}}  + \|\xi_0\|_{L^{2}} \int_{t}^{t+\delta} \|\bm{b}(s)\|_{H^{1+\gamma}} \, ds + \kappa^{\frac{1}{2}} \delta^{\frac{1}{2}}  \|\xi_0\|_{L^{2}}   \bigg]. $$
		Since $\phi$ is arbitrary, the above formula yields
		$$ \|\xi_{t+\delta}-\xi_t\|_{H^{-1}}\lesssim \kappa^{-\frac{\gamma}{2}} \delta^{1-\frac{\gamma}{2}}  \|\xi_0\|^{2}_{L^{2}} + \|\xi_0\|_{L^{2}} \int_{t}^{t+\delta} \|\bm{b}(s)\|_{H^{1+\gamma}} \, ds +   \kappa^{\frac{1}{2}} \delta^{\frac{1}{2}} \|\xi_0\|_{L^{2}} .$$
		Taking the $p$-th moment, we finally get
		\begin{equation}\nonumber
			\begin{split}
				&\quad \ \mathbb{E}\Big[\|\xi_{t+\delta}-\xi_t\|^p_{H^{-1}}\Big] \\
				&\lesssim \kappa^{-\frac{\gamma}{2}p} \, \delta^{(1-\frac{\gamma}{2})p} \|\xi_0\|^{2p}_{L^2}  +\|\xi_0\|^{p}_{L^2} \, \mathbb{E}\bigg[\Big(\int_{t}^{t+\delta} \|\bm{b}(s)\|_{H^{1+\gamma}} \, ds\Big)^p\bigg]+\kappa^{\frac{p}{2}}\delta^{\frac{p}{2}} \|\xi_0\|^{p}_{L^2}  \\
				& \leq \kappa^{-\frac{\gamma}{2}p} \, \delta^{(1-\frac{\gamma}{2})p} \|\xi_0\|^{2p}_{L^2}  +\|\xi_0\|^{p}_{L^2}  \delta^{p-1}  \int_{t}^{t+\delta} \mathbb{E} \big[\|\bm{b}(s)\|^p_{H^{1+\gamma}} \big]ds  + \kappa^{\frac{p}{2}}\delta^{\frac{p}{2}} \|\xi_0\|^{p}_{L^2}  \\
				&\lesssim \kappa^{-\frac{\gamma}{2}p} \, \delta^{(1-\frac{\gamma}{2})p} \|\xi_0\|^{2p}_{L^2}  + \delta^p \nu^{\frac{p}{2}}  \alpha^{\frac{p}{2}} \, C_{\theta,1+\gamma,p} \, \|\xi_0\|^p_{L^{2}}  + \kappa^{\frac{p}{2}}\delta^{\frac{p}{2}} \|\xi_0\|^{p}_{L^2}  .
			\end{split}
		\end{equation}
		Noting that $\kappa>0$ is a fixed parameter, we finish the proof by taking into account our restrictions on $\delta$ and $\alpha$.
	\end{proof}
	
	Based on the conclusion of Lemma \ref{E H^-1}, we can further deduce another useful estimate.
	
	\begin{lemma}\label{E H2-g}
		Let $t\geq 0$, $\delta\in(0,1)$ satisfy $\delta\alpha \gtrsim 1$ and $ \delta^4 \alpha^3 \lesssim 1$, then for every $p\geq 2$ and $\gamma \in (0,1)$, it holds
		$$ \mathbb{E}\Big[\|\xi_{t+\delta}-\xi_t\|^p_{H^{-2-\gamma}}\Big] \lesssim \nu^{p}  \big(  \delta^{\frac{p}{2}} +  \alpha^{-\frac{p}{2}}  \big) C_{\theta,1+\gamma,2p}  \|\xi_{0}\|^{p}_{L^2} \big(1+\|\xi_{0}\|^{p}_{L^2}\big)^2 . $$
	\end{lemma}

	\begin{proof}
		The idea of proof is similar to that of Lemma \ref{E H^-1}, but we divide the right hand side into more terms: for every test function $\phi\in C^{\infty}(\mathbb{T}^2)$,
		\begin{equation}\nonumber
			\begin{split}
				\big| \langle \phi, \xi_{t+\delta}-\xi_t \rangle\big| &\leq \int_{t}^{t+\delta}\big|\langle u_s\cdot \nabla \phi, \xi_s-\xi_t \rangle\big| \,ds + \int_{t}^{t+\delta}\big|\langle u_s\cdot \nabla \phi, \xi_t \rangle\big|\,ds\\
				&\ + \int_{t}^{t+\delta}\big|\langle \bm{b}(s)\cdot \nabla \phi, \xi_s-\xi_t \rangle \big| \,ds+ \Big|\int_{t}^{t+\delta}\langle \bm{b}(s) \cdot \nabla \phi, \xi_t \rangle \, ds \Big|+\kappa\int_{t}^{t+\delta}\big|\langle \Delta\phi, \xi_s \rangle\big| \, ds.
			\end{split}
		\end{equation}
		
		We estimate each term respectively. For the first one,
		\begin{equation}\nonumber
			\begin{split}
				\int_{t}^{t+\delta}\big|\langle u_s\cdot \nabla \phi, \xi_s-\xi_t \rangle\big| \,ds \ &\leq 	\int_{t}^{t+\delta}\|u_s \cdot \nabla \phi\|_{H^{1}} \, \|\xi_s-\xi_t\|_{H^{-1}} \, ds\\
				&\lesssim \int_{t}^{t+\delta}\|\nabla (u_s \cdot \nabla \phi)\|_{L^{2}} \, \|\xi_s-\xi_t\|_{H^{-1}} \, ds,
			\end{split}
		\end{equation}
		where in the second step we have used Poincar\'e's inequality. By Lemma \ref{HHH} and Sobolev embedding theorem, for $\gamma \in (0,1)$, we have the following estimate:
		\begin{equation}\label{L2 decompose}
			\begin{split}
				\|\nabla (u_s \cdot \nabla \phi)\|_{L^{2}} & \leq \|\nabla u_s \cdot \nabla \phi\|_{L^{2}}+ \|u_s\cdot \nabla^2 \phi\|_{L^{2}}\\
				&\lesssim \|\nabla u_s\|_{L^{2}} \|\nabla \phi \|_{L^{\infty}} + \|u_s\|_{H^{1-\gamma}} \|\nabla^2 \phi\|_{H^{\gamma}}\\
				&\lesssim  \|u_s\|_{H^{1}} \|\phi \|_{H^{2+\gamma}} + \|\xi_s\|_{H^{-\gamma}} \| \phi\|_{H^{2+\gamma}}\\
				&\lesssim \|\xi_s\|_{L^{2}} \|\phi \|_{H^{2+\gamma}} ,
			\end{split}
		\end{equation}
		substituting this estimate into the above inequality, we arrive at
		\begin{equation}\nonumber
			\int_{t}^{t+\delta}\big|\langle u_s\cdot \nabla \phi, \xi_s-\xi_t \rangle\big| \, ds \lesssim \|\phi\|_{H^{2+\gamma}}\|\xi_0\|_{L^{2}} \int_{t}^{t+\delta}\|\xi_s-\xi_t\|_{H^{-1}} \,ds.
		\end{equation}
		For the second term, we use Sobolev embedding theorem to get
		\begin{equation}\nonumber
			\begin{split}
				\int_{t}^{t+\delta}\big|\langle u_s\cdot \nabla \phi, \xi_t \rangle\big| \,ds &\leq  \int_{t}^{t+\delta} \|u_s \|_{L^{2}} \|\nabla \phi\|_{L^{\infty}} \|\xi_t\|_{L^{2}} \,ds \\
				&\lesssim \|\phi\|_{H^{2+\gamma}} \|\xi_0\|_{L^{2}}	\int_{t}^{t+\delta} \|\xi_s\|_{H^{-1}}  \, ds\\
				&\lesssim \delta \, \|\phi\|_{H^{2+\gamma}} \|\xi_0\|^2_{L^{2}}.
			\end{split}
		\end{equation}
		Similarly to \eqref{L2 decompose}, we can estimate the third term:
		\begin{equation}\nonumber
			\begin{split}
				\int_{t}^{t+\delta}\big|\langle \bm{b}(s)\cdot \nabla \phi, \xi_s-\xi_t \rangle \big| \, ds
				&\lesssim \int_{t}^{t+\delta}\|\nabla (\bm{b}(s) \cdot \nabla \phi)\|_{L^{2}}\|\xi_s-\xi_t\|_{H^{-1}} \,ds \\
				&\lesssim \|\phi\|_{H^{2+\gamma}} \int_{t}^{t+\delta}\|\bm{b}(s)\|_{H^{1}}  \|\xi_s-\xi_t\|_{H^{-1}} \,ds.
			\end{split}
		\end{equation}
		As for the next term,
		\begin{equation}\nonumber
			\Big|\int_{t}^{t+\delta}\langle \bm{b}(s) \cdot \nabla \phi, \xi_t \rangle \, ds \Big| =\Big| \Big\langle \Big(\int_{t}^{t+\delta}\bm{b}(s) \, ds\Big)\cdot\nabla\phi, \xi_t \Big\rangle\Big| \leq \|\nabla\phi\|_{L^{\infty}} \|\xi_t\|_{L^2} \, \Big\| \int_{t}^{t+\delta}\bm{b}(s) \, ds\Big\|_{L^2} ;
		\end{equation}
		recalling the definition of $\bm{b}$, we have
		$$\int_{t}^{t+\delta} \bm{b}(s) \, ds=2\sqrt{\nu}\sum_{k\in \mathbb{Z}_0^2}\theta_k\sigma_k \big(W_{t+\delta}^k-W_t^k\big)-\alpha^{-1} \big(\bm{b}(t+\delta)-\bm{b}(t)\big),$$
		and therefore, for $\gamma \in (0,1)$, we can get
		\begin{equation}\nonumber
			\begin{split}
				\Big|\int_{t}^{t+\delta}\langle \bm{b}(s) \cdot \nabla \phi, \xi_t \rangle \, ds \Big| &
				\lesssim \nu^{\frac12} \|\phi\|_{H^{2+\gamma}}\|\xi_0\|_{L^2} \bigg[\sum_{k\in \mathbb{Z}_0^2}\theta_k^2  \big(W^k_{t+\delta}-W^k_t \big)^2 \bigg]^{\frac{1}{2}}\\
				&\quad + \alpha^{-1} \, \|\phi\|_{H^{2+\gamma}} \|\xi_0\|_{L^2}  \big(\|\bm{b}(t+\delta)\|_{L^{2}}+\|\bm{b}(t)\|_{L^{2}} \big).			
			\end{split}
		\end{equation}
		Finally, the fifth term can be estimated as follows:
		$$	\kappa\int_{t}^{t+\delta}\big|\langle \Delta\phi, \xi_s \rangle\big| \, ds\leq \kappa \int_{t}^{t+\delta} \|\Delta \phi\|_{L^2} \|\xi_s\|_{L^2} \, ds \lesssim \kappa\delta \, \|\phi\|_{H^{2+\gamma}}\|\xi_0\|_{L^2}.$$
		Combining these results together and noticing the arbitrariness of $\phi$, we arrive at
		\begin{equation}\label{lemma3.2 decomposition}
			\begin{split}
				\|\xi_{t+\delta}-\xi_t\|_{H^{-2-\gamma}} &\lesssim \|\xi_0\|_{L^{2}} \int_{t}^{t+\delta}\|\xi_s-\xi_t\|_{H^{-1}} \, ds + \delta \, \|\xi_0\|^2_{L^{2}} \\
				&\quad +\int_{t}^{t+\delta}\|\bm{b}(s)\|_{H^{1}}  \|\xi_s-\xi_t\|_{H^{-1}}\, ds  + \nu^{\frac12} \|\xi_0\|_{L^2} \bigg[\sum_{k\in \mathbb{Z}_0^2}\theta_k^2  \big(W^k_{t+\delta}-W^k_t \big)^2 \bigg]^{\frac{1}{2}} \\
				&\quad+\alpha^{-1} \|\xi_0\|_{L^2}\big(\|\bm{b}(t+\delta)\|_{L^{2}}+\|\bm{b}(t)\|_{L^{2}} \big) +\kappa\delta \, \|\xi_0\|_{L^2}.
			\end{split}
		\end{equation}
		
		To complete the proof, we also need the following several estimates. By Lemma \ref{E H^-1},
		\begin{equation}\label{E1}
			\begin{split}
				\mathbb{E} \bigg[\Big(\int_{t}^{t+\delta}\|\xi_s-\xi_t\|_{H^{-1}} ds\Big)^p \bigg] &\leq  \delta^{p-1}  \int_{t}^{t+\delta} \mathbb{E}\big[ \|\xi_s-\xi_t\|^p_{H^{-1}} \big]ds \\
				&\lesssim \delta^{2p} \nu^{\frac{p}{2}}  \alpha^{\frac{p}{2}}  \, C_{\theta,1+\gamma,p} \, \|\xi_{0}\|^{p}_{L^2}\big(1+\|\xi_{0}\|^{p}_{L^2}\big) .
			\end{split}
		\end{equation}
		H\"older's inequality yields
		\begin{equation}\label{E2}
			\begin{split}
				&\quad \ \mathbb{E}\bigg[\Big(	\int_{t}^{t+\delta}\|\bm{b}(s)\|_{H^{1}}  \|\xi_s-\xi_t\|_{H^{-1}} \, ds \Big)^p \bigg] \\
				& \leq \mathbb{E} \bigg[\delta^{p-1} \int_{t}^{t+\delta}\|\bm{b}(s)\|^p_{H^{1}}  \|\xi_s-\xi_t\|^p_{H^{-1}} \, ds \bigg] \\
				&\leq  \delta^{p-1}\int_{t}^{t+\delta}\Big(\mathbb{E}\big[ \|\bm{b}(s)\|^{2p}_{H^{1}}\big]\Big)^{\frac{1}{2}} \, \Big(\mathbb{E} \big[ \|\xi_s-\xi_t\|^{2p}_{H^{-1}}\big]\Big) ^{\frac{1}{2}} \, ds\\
				&\lesssim  \delta^{2p} \nu^p \alpha^p \, C_{\theta,1+\gamma,2p} \, \|\xi_0\|^p_{L^2} \big(1+\|\xi_{0}\|^{p}_{L^2}\big)  .
			\end{split}
		\end{equation}
		Besides, by Jensen's inequality, for $p\geq 2$, the following formula holds:
		\begin{equation}\label{E3}
			\mathbb{E} \bigg[  \Big(\sum_{k\in \mathbb{Z}_0^2}\theta_k^2 \, \big(W^k_{t+\delta}-W^k_t \big)^2 \Big)^{\frac{p}{2}}  \, \bigg] \leq \mathbb{E} \bigg[ \sum_{k\in \mathbb{Z}_0^2}\theta_k^2 \, \big|W^k_{t+\delta}-W^k_t\big|^p \bigg] \lesssim \sum_{k\in \mathbb{Z}_0^2}\theta_k^2 \, \delta^{\frac{p}{2}} = \delta^{\frac{p}{2}}.
		\end{equation}
		According to the definition of $\bm{b}(t)$, for $t\geq 0$, we have
		\begin{equation}\label{E4}
			\mathbb{E} \Big[\|\bm{b}(t+\delta)\|^p_{L^{2}} \Big]=\mathbb{E} \Big[\|\bm{b}(t)\|^p_{L^{2}} \Big] \leq (4\nu)^{\frac{p}{2}} \sum_{k \in \mathbb{Z}_0^2} \theta^2_k \, \mathbb{E} \Big[\big|\eta^{\alpha,k}(t)\big|^p \Big]\lesssim \nu^{\frac{p}{2}} \alpha^{\frac{p}{2}}.
		\end{equation}
		Inserting \eqref{E1}-\eqref{E4} into \eqref{lemma3.2 decomposition}, we can easily get
		\begin{equation}\nonumber
			\begin{split}
				\mathbb{E}\Big[\|\xi_{t+\delta}-\xi_t\|^p_{H^{-2-\gamma}}\Big] &\lesssim  \delta^{2p} \nu^{\frac{p}{2}}  \alpha^{\frac{p}{2}} \, C_{\theta,1+\gamma,p} \,  \|\xi_{0}\|^{2p}_{L^2} \big(1+\|\xi_{0}\|^{p}_{L^2}\big) + \delta^p \, \|\xi_{0}\|^{2p}_{L^2} \\
				&\quad +  \delta^{2p} \nu^p \alpha^p \, C_{\theta,1+\gamma,2p} \, \|\xi_0\|^p_{L^2}\big(1+\|\xi_{0}\|^{p}_{L^2}\big) + \delta^{\frac{p}{2}} \nu^{\frac{p}{2}}   \, \|\xi_{0}\|^{p}_{L^2}  \\
				&\quad+ \nu^{\frac{p}{2}} \alpha^{-\frac{p}{2}} \, \|\xi_{0}\|^{p}_{L^2}  +\kappa^p \delta^p \, \|\xi_{0}\|^{p}_{L^2}.
			\end{split}
		\end{equation}
		It is clear that the parts involving the $L^2$-norm of initial data are dominated by $\|\xi_0\|^p_{L^2} \big(1+ \|\xi_{0}\|^{p}_{L^2}\big)^2$; the conclusion follows by noticing our previous assumptions on the parameters.
	\end{proof}
	
	The following result is analogous to \cite[Proposition 3.3]{Pappalettera} where a similar estimate, against test functions, was proved for the solution $h$ of \eqref{pappa-eq-1}. We remark that the stronger estimate as below is needed at the end of the proof of Proposition \ref{estimate on f}. We first divide the interval $[0,T]$ into many subintervals of the form $ [n\delta, (n+1)\delta]$, $n\in \mathbb{N}$, where $\delta\in(0,1)$ is a small parameter such that $T/\delta$ is an integer, then we estimate the quantity in each interval of length $\delta$, and finally sum them up.
	
	\begin{proposition}\label{main proposition}
		Fix $\beta>3$, $\gamma\in(0,\frac{1}{3})$, then there exist $\epsilon>0$, $\delta\in (0,1)$ and $\rho\in(0,\frac{1}{4})$ such that for $\alpha$ large enough, the following estimate holds:	
		\begin{equation} \nonumber
			\begin{split}
				&\mathbb{E}\bigg[\sup_{1\leq m<n \leq T/\delta-1} \frac1{(|n-m|\delta)^\rho } \Big\|\xi_{n\delta}-\xi_{m\delta}-(\kappa+\nu)\int_{m\delta}^{n\delta} \Delta \xi_s \, ds+\int_{m\delta}^{n\delta} u_s\cdot \nabla \xi_s \, ds \Big\|_{H^{-\beta}} \bigg]\\
				&\quad \lesssim T \|\xi_0\|_{L^2} \big(1+\|\xi_0\|_{L^{2}}\big)^2 \big(\nu^{1+\frac{\gamma}{2}}\alpha^{-\epsilon}  +\nu^{\frac{1}{2}}\|\theta\|_{\ell^{\infty}}\big).
			\end{split}
		\end{equation}
	\end{proposition}
	
	As the proof of Proposition \ref{main proposition} is very long, we postpone it to Section 5. We mention that some restrictions on the parameters $\alpha$ and $\delta$ are necessary in order to obtain a sufficiently small estimate, and one can find the specific details in Section \ref{subs-proof-Prop-3.3}.

	\section{Proofs of main results}
	
	This section consists of two parts which are devoted to the proofs of Theorems \ref{thm1} and \ref{thm2}, respectively.
	
	\subsection{Proof of Theorem \ref{thm1}}
	To prove Theorem \ref{thm1}, we first define a random distribution $f$ as follows:
	\begin{equation}\label{def of f}
		f_t =\xi_{t}-\xi_{0}-(\kappa+\nu)\int_{0}^{t} \Delta\xi_{s} \, ds +\int_{0}^{t} u_{s} \cdot \nabla \xi_{s}\, ds.
	\end{equation}
	If we replace $\xi$ by $\bar\xi$ and $u$ by $\bar u$, then the right-hand side vanishes; since we expect that $\xi$ is close to $\bar\xi$, the distribution $f$ would be small in suitable norms. We first prove a regularity estimate on $f$, which will be used in the proof of Proposition \ref{estimate on f}.
	
	\begin{lemma}\label{f Hbeta}
		For every $0\leq s<t \leq T$ and $\gamma>0$, it holds	
		\begin{equation}\nonumber
			\|f_t-f_s\|_{H^{-2}} \lesssim \|\xi_0\|^2_{L^{2}} \,  |t-s|+ \|\xi_0\|_{L^{2}} \int_{s}^{t} \|\bm{b}(r)\|_{H^{\gamma}} \, dr +(\kappa+\nu) \, \|\xi_0\|_{L^2}  \,|t-s| .
		\end{equation}
	\end{lemma}

	\begin{proof}
		By \eqref{SPDE}, for every $s,t\in[0,T]$ and $s<t$, we have
		\begin{equation}\nonumber
			\xi_t -\xi_s =-\int_{s}^{t} u_r \cdot \nabla \xi_r \, dr -\int_{s}^{t} \bm{b}(r) \cdot \nabla \xi_r \, dr +\kappa \int_{s}^{t} \Delta \xi_r \, dr.
		\end{equation}
		Then we can further get
		\begin{equation}\label{H-2}
			\begin{split}
				\|\xi_t-\xi_s\|_{H^{-2}}&\leq \Big\|\int_{s}^{t} u_r \cdot \nabla \xi_r \, dr \Big\|_{H^{-2}} + \Big\|\int_{s}^{t} \bm{b}(r) \cdot \nabla \xi_r \, dr \Big\|_{H^{-2}} + \kappa \, \Big\|\int_{s}^{t} \Delta \xi_r \,dr\Big\|_{H^{-2}}\\
				&\lesssim \int_{s}^{t} \|u_r \cdot \nabla \xi_r \|_{H^{-2}} \, dr+ \int_{s}^{t}\|\bm{b}(r) \cdot \nabla \xi_r\|_{H^{-2}} \, dr +\kappa\int_{s}^{t} \| \xi_0\|_{L^{2}} \, dr.
			\end{split}
		\end{equation}	
		
		Now we will estimate the first and the second terms respectively. Using the divergence free property of $u$, we have
		\begin{equation}\nonumber
			\|u_r \cdot \nabla \xi_r\|_{H^{-2}} = \|\nabla \cdot (\xi_r u_r)\|_{H^{-2}}  \lesssim \|\xi_r u_r\|_{H^{-1}} ;
		\end{equation}
		besides, by H\"older's inequality and Sobolev embedding theorem, for $\phi \in C^{\infty}(\T^2)$,
		\begin{equation}\nonumber
			\big|\langle \xi_r u_r, \phi\rangle \big| \leq \|\xi_r\|_{L^{2}} \|u_r\|_{L^{4}} \|\phi\|_{L^{4}} \lesssim \|\xi_0\|_{L^{2}}  \|u_r\|_{H^{\frac{1}{2}}}  \|\phi\|_{H^{\frac{1}{2}}} \lesssim \|\xi_0\|^2_{L^{2}} \|\phi\|_{H^{1}} .
		\end{equation}
		Then we get $\|\xi_r u_r\|_{H^{-1}} \lesssim \|\xi_0\|^2_{L^{2}} $ and thus $\|u_r \cdot \nabla \xi_r\|_{H^{-2}} \lesssim \|\xi_0\|^2_{L^{2}}$.
		
		As for the next term, notice that $\bm{b}(r)$ is divergence free, we use the same method as above to estimate it: for $\gamma>0$, we have
		\begin{equation}\nonumber	
			\|\bm{b}(r) \cdot \nabla \xi_r\|_{H^{-2}} \lesssim \|\xi_r \, \bm{b}(r) \|_{H^{-1}} \lesssim \|\xi_0\|_{L^{2}} \|\bm{b}(r)\|_{H^{\gamma}}.
		\end{equation}
		Having the above results at hand, we combine \eqref{H-2} with \eqref{def of f} and arrive at
			\begin{equation}\nonumber
			\begin{split}
				\|f_t-f_s\|_{H^{-2}} &\leq \|\xi_t-\xi_s\|_{H^{-2}} + (\kappa+\nu) \, \int_{s}^{t} \| \Delta \xi_r\|_{H^{-2}}  \, dr + \int_{s}^{t} \|u_r \cdot \nabla \xi_r \|_{H^{-2}}  \, dr \\
				&\leq \int_{s}^{t} \|u_r \cdot \nabla \xi_r \|_{H^{-2}} \, dr+ \int_{s}^{t}\|\bm{b}(r) \cdot \nabla \xi_r\|_{H^{-2}} \, dr + (\kappa+\nu) \int_{s}^{t} \|\xi_0\|_{L^2} \, dr\\
				&\lesssim \|\xi_0\|^2_{L^{2}} \, |t-s|  + \|\xi_0\|_{L^{2}} \int_{s}^{t} \|\bm{b}(r)\|_{H^{\gamma}} \, dr +  (\kappa+\nu)  \|\xi_0\|_{L^{2}} \, |t-s| .
			\end{split}
		\end{equation}
	\end{proof}

	\begin{proposition}\label{estimate on f}
		Let $\beta>3$, $\gamma \in (0, \frac{1}{3})$ and $T\geq 1$, then there are $\rho\in(0,\frac{1}{4})$ and $\epsilon>0$ such that for every $\alpha$ sufficiently large, it holds
		\begin{equation}\nonumber
			\mathbb{E}\Big[\|f\|_{C^{\rho}([0,T],H^{-\beta})}\Big] \lesssim T \|\xi_0\|_{L^2} \big(1+\|\xi_0\|_{L^{2}}\big)^2 \big(\nu^{1+\frac{\gamma}{2}}\alpha^{-\epsilon}  +\nu^{\frac{1}{2}}\|\theta\|_{\ell^{\infty}}\big).
		\end{equation}
	\end{proposition}
	
	\begin{proof}
		First, as $f_0=0$, we give an equivalent norm as follows:
		\begin{equation}\label{equivalent seminorm}
			\|f\|_{C^{\rho}([0,T],H^{-\beta})} \sim 	\sup_{0\leq s<t \leq T} \frac{\|f_t-f_s\|_{H^{-\beta}} }{|t-s|^{\rho}}.
		\end{equation}
		Then we will prove Proposition \ref{estimate on f}  in the following two cases.
		
		\textbf{Case 1:} $|t-s| \leq \delta$. By H\"older's inequality, for $\rho \in (0,\frac{1}{4})$, we have
		\begin{equation}\nonumber
			\int_{s}^{t}\|\bm{b}(r)\|_{H^{\gamma}} \,dr \leq \bigg(\int_{s}^{t}1 \, dr\bigg)^{1-\rho} \bigg(\int_{s}^{t} \|\bm{b}(r)\|^{\frac{1}{\rho}}_{H^{\gamma}} \, dr\bigg)^{\rho} \leq |t-s|^{1-\rho} \bigg(\int_{s}^{t} \|\bm{b}(r)\|^{\frac{1}{\rho}}_{H^{\gamma}} \, dr\bigg)^{\rho}.
		\end{equation}
		Furthermore, we can get
		\begin{equation}\nonumber
			\mathbb{E}\Bigg[\sup_{\substack{0\leq s<t \leq T \\|t-s|\leq\delta}}
			\frac{\int_{s}^{t}\|\bm{b}(r)\|_{H^{\gamma}} \, dr }{|t-s|^{\rho}}\Bigg]\leq \delta^{1-2\rho} \, \mathbb{E} \bigg(\int_{0}^{T} \|\bm{b}(r)\|^{\frac{1}{\rho}}_{H^{\gamma}} \, dr\bigg)^{\rho} \lesssim  \delta^{1-2\rho} \nu^{\frac{1}{2}} \alpha^{\frac{1}{2}} \, T^{\rho} C^{\rho}_{\theta,\gamma,1/\rho} \, ,
		\end{equation}
		where in the last step we have used Lemma \ref{lemma Eb}. Then for $\beta>3$, Lemma \ref{f Hbeta} yields
		\begin{equation}\label{Case 1 result}
			\begin{split}
				\mathbb{E}\Bigg[\sup_{\substack{0\leq s<t \leq T \\ |t-s|\leq\delta}}\frac{\|f_t-f_s\|_{H^{-\beta}} }{|t-s|^{\rho}}\Bigg] &\lesssim \delta^{1-\rho}  \, \|\xi_0\|^2_{L^{2}}  + \delta^{1-2\rho} \nu^{\frac{1}{2}} \alpha^{\frac{1}{2}} \, T^{\rho} C^{\rho}_{\theta,\gamma,1/\rho}  \,  \|\xi_0\|_{L^{2}}  + (\kappa+\nu) \,\delta^{1-\rho} \, \|\xi_0\|_{L^{2}}  \\
				&\lesssim \delta^{1-2\rho} \nu^{\frac{1}{2}} \alpha^{\frac{1}{2}} \, T^{\rho} C^{\rho}_{\theta,\gamma,1/\rho}  \, \|\xi_0\|_{L^{2}} \big(1+\|\xi_0\|_{L^{2}}\big) .
			\end{split}
		\end{equation}
		
		\textbf{Case 2:} $|t-s| > \delta$. We suppose $s \in [(m-1)\delta, m\delta)$ and $t\in(n\delta, (n+1)\delta]$, where $n,m \in \mathbb{N}$ and $m\leq n$. Hence, if $n>m$, we have
		\begin{equation}\nonumber
			\frac{\|f_t-f_s\|_{H^{-\beta}} }{|t-s|^{\rho}} \leq \frac{\|f_t-f_{n\delta}\|_{H^{-\beta}} }{|t-n\delta|^{\rho}}+\frac{\|f_{n\delta}-f_{m\delta}\|_{H^{-\beta}} }{|n\delta-m\delta|^{\rho}}+\frac{\|f_{m\delta}-f_s\|_{H^{-\beta}} } {|m\delta-s|^{\rho}},
		\end{equation}
		while for $n=m$, the following formula holds:
		\begin{equation}\nonumber
			\frac{\|f_t-f_s\|_{H^{-\beta}} }{|t-s|^{\rho}} \leq \frac{\|f_t-f_{n\delta}\|_{H^{-\beta}} }{|t-n\delta|^{\rho}}+\frac{\|f_{n\delta}-f_s\|_{H^{-\beta}} }{|n\delta-s|^{\rho}}.
		\end{equation}
		
		Then we can combine the above two cases and get
		\begin{equation}\label{combination of 2 cases}
			\sup_{0\leq s<t \leq T} \frac{\|f_t-f_s\|_{H^{-\beta}} }{|t-s|^{\rho}} \lesssim \sup_{\substack{0\leq s<t \leq T \\ |t-s|\leq\delta}} \frac{\|f_t-f_s\|_{H^{-\beta}} }{|t-s|^{\rho}} +\sup_{1 \leq m<n \leq T/\delta-1} \frac{\|f_{n\delta}-f_{m\delta}\|_{H^{-\beta}} }{|n\delta-m\delta|^{\rho}}.
		\end{equation}
		By the definition of $f$ and Proposition \ref{main proposition}, for every $\beta>3$, it holds
		\begin{equation}\label{Case 2 result}
			\mathbb{E}\bigg[\sup_{1\leq m<n \leq T/\delta-1} \frac{\|f_{n\delta}-f_{m\delta}\|_{H^{-\beta}} }{|n\delta-m\delta|^{\rho}}\bigg] \lesssim T \|\xi_0\|_{L^2} \big(1+\|\xi_0\|_{L^{2}}\big)^2 \big(\nu^{1+\frac{\gamma}{2}}\alpha^{-\epsilon}  +\nu^{\frac{1}{2}}\|\theta\|_{\ell^{\infty}}\big) .
		\end{equation}
		Taking into account \eqref{Case 1 result}--\eqref{Case 2 result} and noticing the restrictions on the parameters in Proposition \ref{main proposition}, we use \eqref{equivalent seminorm} to complete the proof of Proposition \ref{estimate on f}.
	\end{proof}

	Now we give the following proposition which indicates that $f$ is bounded in $H^{-1}$.
	\begin{proposition}\label{sup f H-1}
		Suppose $f$ is defined as in \eqref{def of f}, then for $\gamma \in (0,1)$ and $T\geq 1$, it holds	
		\begin{equation}\nonumber
			\sup_{t\in [0,T]} \|f_t\|_{H^{-1}} \lesssim  T^{1-\frac{\gamma}{2}}  \|\xi_0\|_{L^{2}} \big(1+ \|\xi_0\|_{L^{2}}\big) \big(\nu \kappa^{-\frac{1}{2}}+\kappa^{-\frac{\gamma}{2}}\big) .
		\end{equation}
	\end{proposition}
	
	\begin{proof}
		According to \eqref{def of f}, the following formula holds:
		\begin{equation}\label{f H-1 norm}
			\begin{split}
				\|f_t\|_{H^{-1}} &\leq \|\xi_t-\xi_0\|_{H^{-1}} + (\kappa+\nu) \int_{0}^{t} \|\Delta \xi_s\|_{H^{-1}} \, ds +\int_{0}^{t} \|u_s \cdot \nabla \xi_s\|_{H^{-1}} \, ds\\
				&\lesssim 2\|\xi_0\|_{L^{2}} + (\kappa+\nu) \int_{0}^{t} \|\xi_s\|_{H^{1}} \, ds + \int_{0}^{t} \|u_s \cdot \nabla \xi_s\|_{H^{-1}} \, ds.
			\end{split}
		\end{equation}
		Applying H\"older's inequality and \eqref{priori estimates 1}, we have
		\begin{equation}\label{prop 5.3 1st}
			\int_{0}^{t} \|\xi_s\|_{H^{1}}\, ds \leq \Big(\int_{0}^{t} 1 \, ds\Big)^{\frac{1}{2}} \Big(\int_{0}^{t} \|\xi_s\|^2_{H^{1}} \, ds\Big)^{\frac{1}{2}} \lesssim \kappa^{-\frac{1}{2}} \, t^{\frac{1}{2}} \, \|\xi_0\|_{L^{2}} .
		\end{equation}
		
		As for the last term, notice that for $\gamma \in (0,1)$, the following inequality holds:
		\begin{equation}\nonumber
			\|u_s \cdot \nabla \xi_s\|_{H^{-1}} \lesssim \| \xi_s \, u_s \|_{L^{2}} \leq \|\xi_s\|_{L^{2}} \, \|u_s\|_{L^{\infty}} \lesssim \|\xi_0\|_{L^{2}} \, \|\xi_s\|_{H^{\gamma}} ,
		\end{equation}
		then we can use Lemma \ref{interpolation} and \eqref{priori estimates 1} to further get
		\begin{equation}\label{prop 5.3 2nd}
			\begin{split}
				\int_{0}^{t} \|u_s \cdot \nabla \xi_s\|_{H^{-1}} \, ds &\lesssim \|\xi_0\|_{L^{2}} \int_{0}^{t} \|\xi_s\|^{\gamma}_{H^{1}} \, \|\xi_s\|^{1-\gamma}_{L^{2}} \, ds\\
				&\leq \|\xi_0\|^{2-\gamma}_{L^{2}} \Big(\int_{0}^{t} 1 \, ds\Big)^{1-\frac{\gamma}{2}} \Big(\int_{0}^{t} \|\xi_s\|^2_{H^{1}} \, ds\Big)^{\frac{\gamma}{2}} \\
				&\lesssim \kappa^{-\frac{\gamma}{2}} \, t^{1-\frac{\gamma}{2}} \, \|\xi_0\|^{2}_{L^{2}} .
			\end{split}
		\end{equation}
		
		Inserting \eqref{prop 5.3 1st} and \eqref{prop 5.3 2nd} into \eqref{f H-1 norm},  we take supremum and deduce
		\begin{equation}\nonumber
			\sup_{t\in [0,T]} \|f_t\|_{H^{-1}} \lesssim2\|\xi_0\|_{L^{2}} + (\kappa^{\frac{1}{2}}+\nu \kappa^{-\frac{1}{2}}) \, T^{\frac{1}{2}} \|\xi_0\|_{L^{2}} + \kappa^{-\frac{\gamma}{2}} T^{1-\frac{\gamma}{2}}  \|\xi_0\|^{2}_{L^{2}} .
		\end{equation}
		The proposition follows due to the choices of parameters.
	\end{proof}

	With the above preparations in mind, we can prove the first main theorem now.
	
	\begin{proof}[Proof of Theorem \ref{thm1}]
		By the definition of $f$, we have
		\begin{equation}\nonumber
			\xi_t=\xi_0+ (\kappa+\nu) \int_{0}^{t} \Delta \xi_s \, ds - \int_{0}^{t} u_s \cdot \nabla \xi_s \, ds +f_t,
		\end{equation}
		while by \eqref{PDE}, it holds
		\begin{equation}\nonumber
			\bar{\xi}_t=\xi_0+(\kappa+\nu) \int_{0}^{t} \Delta \bar{\xi}_s \, ds - \int_{0}^{t} \bar{u}_s \cdot \nabla \bar{\xi}_s \, ds .
		\end{equation}
		Define
		$$X_t:=f_t-\int_{0}^{t}(u_s\cdot\nabla \xi_s -\bar{u}_s \cdot \nabla \bar{\xi}_s ) \, ds,$$
		then the difference $\xi -\bar{\xi}$ satisfies
		\begin{equation}\label{difference of 2 solutions}
			\xi_t-\bar{\xi}_t=(\kappa+\nu) \int_{0}^{t} \Delta (\xi_s-\bar{\xi}_s) ds + X_t.
		\end{equation}
		
		We first prove the theorem in the Sobolev spaces $H^{-\tilde{\vartheta}}$ with $\tilde{\vartheta}>1$; without loss of generality we can assume $\tilde{\vartheta} \in (1,\frac{3}{2})$. Recall that we have already obtained in Proposition \ref{estimate on f} an estimate on $\|f\|_{C^{\rho}([0,T],H^{-\beta})} $ for $\beta>3$; we shall fix such a $\beta$ in the sequel. Besides, Proposition \ref{sup f H-1} gives us a bound on $\sup_{t\in [0,T]}\|f_t\|_{H^{-1}}$, hence by Lemma \ref{interpolation}, for any $\tilde{\vartheta} \in (1,\frac{3}{2})$, there exists $\zeta \in (0,1)$ satisfying $\beta\zeta +(1-\zeta )=\tilde{\vartheta}$, such that
		$$\|f_t-f_s\|_{H^{-\tilde{\vartheta}}} \leq \|f_t-f_s\|^{1-\zeta}_{H^{-1}} \, \|f_t-f_s\|^{\zeta }_{H^{-\beta}} .$$
		Furthermore, for $\rho \in (0,\frac{1}{4})$, we can calculate $\|f\|_{C^{\rho\zeta}([0,T],H^{-\tilde{\vartheta}})} $ as
		\begin{equation}\label{thm1 interpolation}
			\begin{split}
				\|f\|_{C^{\rho\zeta}([0,T],H^{-\tilde{\vartheta}})} & \sim  \sup_{0\leq s<t \leq T} \frac{\|f_t-f_s\|_{H^{-\tilde{\vartheta}}} }{|t-s|^{\rho\zeta}} \\
				&\leq \sup_{0\leq s<t \leq T} \frac{\|f_t-f_s\|^{1-\zeta}_{H^{-1}} \, \|f_t-f_s\|^{\zeta }_{H^{-\beta}} }{|t-s|^{\rho \zeta}}\\
				& \lesssim \Big(\sup_{t\in [0,T]} \|f_t\|^{1-\zeta }_{H^{-1}} \Big) \,  \|f\|^{\zeta }_{C^{\rho}([0,T],H^{-\beta})} .
			\end{split}
		\end{equation}
Proposition \ref{estimate on f} implies that, $\P$-a.s., $f\in C^{\rho\zeta}([0,T],H^{-\tilde{\vartheta}})$. Next, we have
  $$\aligned \Big\|\int_{0}^{t} u_r\cdot\nabla \xi_r\,d r - \int_{0}^{s} u_r\cdot\nabla \xi_r\,d r \Big\|_{H^{-\tilde{\vartheta}}} \le \int_s^t \|u_r\cdot\nabla \xi_r \|_{H^{-\tilde{\vartheta}}} \, dr \lesssim \int_s^t \|u_r \xi_r \|_{H^{1-\tilde{\vartheta}}} \, dr
  \endaligned $$
and by Lemma \ref{HHH}, $\|u_r \xi_r \|_{H^{1-\tilde{\vartheta}}} \lesssim \|\xi_r \|_{L^2} \|u_r \|_{H^{2-\tilde{\vartheta}}} \lesssim \|\xi_r \|_{L^2}^2 \le \|\xi_0 \|_{L^2}^2$; thus, the function $t\mapsto \int_{0}^{t} u_r\cdot\nabla \xi_r\,d r\in H^{-\tilde{\vartheta}}$ is Lipschitz continuous. With slightly more effort, one can show that it actually belongs to $C^1([0,T],H^{-\tilde{\vartheta}})$ by using the fact $\xi\in C([0,T], L^2)$; similar result holds for $t\mapsto \int_{0}^{t} \bar u_r\cdot\nabla \bar\xi_r\,d r\in H^{-\tilde{\vartheta}}$.

Summarizing the above arguments, for every $\tilde{\vartheta} \in (1,\frac{3}{2})$ and $\rho\in (0,\frac{1}{4})$, we deduce that, $\P$-a.s., $X \in C^{\rho\zeta}([0,T],H^{-\tilde{\vartheta}})$ for all $T\geq 1$, where $\zeta\in (0,1)$ is defined as above.
		Therefore, according to \cite[Theorem 1]{Young integral}, there exists a linear map $\mathcal{G}$, such that
		\begin{equation}\label{linear decompose}
			\xi-\bar{\xi}=\mathcal{G}(X)=\mathcal{G}(f)-\mathcal{G}\Big(\int_{0}^{\cdot}(u_s\cdot\nabla \xi_s -\bar{u}_s \cdot \nabla \bar{\xi}_s ) \, ds\Big);
		\end{equation}
		furthermore,  the following result holds:
		\begin{equation}\label{prf of thm 1 1st conclusion}
			\sup_{t\in [0,T]}\|\mathcal{G}(f_t)\|_{H^{-\tilde{\vartheta}}} \lesssim \|f\|_{C^{\rho\zeta}([0,T],H^{-\tilde{\vartheta}})} .
		\end{equation}

		Now we will deal with the last term in \eqref{linear decompose}. As $\int_{0}^{\cdot}(u_s\cdot\nabla \xi_s -\bar{u}_s \cdot \nabla \bar{\xi}_s ) \, ds$ belongs to $C^1([0,T],H^{-\tilde{\vartheta}})$, then by \cite[Theorem 1]{Young integral},  it holds
		\begin{equation}\label{nonlinear decompose}
			\begin{split}
				&\quad \ \mathcal{G}\Big(\int_{0}^{\cdot}(u_s\cdot\nabla \xi_s -\bar{u}_s \cdot \nabla \bar{\xi}_s ) \, ds \Big)(t)\\
				&=\int_{0}^{t} e^{(\kappa+\nu)(t-s)\Delta} (u_s\cdot\nabla \xi_s -\bar{u}_s \cdot \nabla \bar{\xi}_s ) \, ds\\
				&=\int_{0}^{t} e^{(\kappa+\nu)(t-s)\Delta} \big[(u_s-\bar{u}_s)\cdot \nabla \xi_s \big] ds+ \int_{0}^{t} e^{(\kappa+\nu)(t-s)\Delta} \big[\bar{u}_s \cdot \nabla (\xi_s - \bar{\xi}_s)\big]ds.
			\end{split}
		\end{equation}
		For the first term, we can use the standard heat kernel estimate (see e.g. \cite[Section 2]{FLD quantitative}) to get
		\begin{equation}\label{quantitative lemma}
			\Big\|\int_{0}^{t} e^{(\kappa+\nu)(t-s)\Delta} \big[(u_s-\bar{u}_s)\cdot \nabla \xi_s \big] ds\Big\|^2_{H^{-\tilde{\vartheta}}} \lesssim \frac1{\kappa+\nu} \int_{0}^{t} \| (u_s-\bar{u}_s)\cdot \nabla \xi_s \|^2_{H^{-\tilde{\vartheta}-1}} \, ds.
		\end{equation}
		Noting that $\| (u_s-\bar{u}_s)\cdot \nabla \xi_s \|_{H^{-\tilde{\vartheta}-1}}=\| (u_s-\bar{u}_s) \, \xi_s \|_{H^{-\tilde{\vartheta}}}$, then for $\tilde{\vartheta} \in (1,\frac{3}{2})$ and $\phi\in C^\infty(\T^2)$, Lemma \ref{HHH} yields
		\begin{equation}\nonumber
			\big|\langle (u_s-\bar{u}_s)\, \xi_s , \phi \rangle \big| \leq \|u_s-\bar{u}_s\|_{H^{1-\tilde{\vartheta}}} \|\xi_s  \phi \|_{H^{\tilde{\vartheta}-1}} \lesssim \|\xi_s-\bar{\xi}_s\|_{H^{-\tilde{\vartheta}}} \|\xi_s\|_{H^{\frac{1}{2}}} \|\phi\|_{H^{\tilde{\vartheta}-\frac{1}{2}}}.
		\end{equation}
		Therefore,  by duality of Sobolev norms and Lemma \ref{interpolation},
		\begin{equation}\nonumber
			\| (u_s-\bar{u}_s)\, \xi_s \|_{H^{-\tilde{\vartheta}}} \lesssim \|\xi_s-\bar{\xi}_s\|_{H^{-\tilde{\vartheta}}} \|\xi_s\|_{H^{\frac{1}{2}}} \lesssim \|\xi_s-\bar{\xi}_s\|_{H^{-\tilde{\vartheta}}} \|\xi_s\|^{\frac{1}{2}}_{H^{1}} \|\xi_0\|^{\frac{1}{2}}_{L^{2}}.
		\end{equation}
		Applying the above result to \eqref{quantitative lemma}, we obtain
		\begin{equation}\label{nonlinear 1st term}
			\Big\|\int_{0}^{t} e^{(\kappa+\nu)(t-s)\Delta} \big[(u_s-\bar{u}_s)\cdot \nabla \xi_s \big] ds\Big\|^2_{H^{-\tilde{\vartheta}}}  \lesssim \frac1{\kappa+\nu} \int_{0}^{t} \|\xi_s-\bar{\xi}_s\|^2_{H^{-\tilde{\vartheta}}}  \|\xi_0\|_{L^{2}} \|\xi_s\|_{H^{1}} \, ds .
		\end{equation}	
		
		The latter term of \eqref{nonlinear decompose} can be treated similarly as follows:
		\begin{equation}\label{quantitative lemma 2}
			\Big\|\int_{0}^{t} e^{(\kappa+\nu)(t-s)\Delta} \big[\bar{u}_s \cdot \nabla( \xi_s -\bar{\xi}_s) \big] ds\Big\|^2_{H^{-\tilde{\vartheta}}} \lesssim \frac1{\kappa+\nu} \int_{0}^{t} \big\|\bar{u}_s \, (\xi_s - \bar{\xi}_s)  \big\|^2_{H^{-\tilde{\vartheta}}} \, ds.
		\end{equation}
		Let $\phi\in C^\infty(\T^2)$ be a test function; for any fixed $s \in [0,T]$, we denote $A(\bar{u}_s \phi ):=\int_{\mathbb{T}^2} (\bar{u}_s \phi)(x) \, dx$. As $\int_{\mathbb{T}^2} (\xi_s - \bar{\xi}_s) (x) \, dx=0$, we have $\langle  \xi_s - \bar{\xi}_s, A(\bar{u}_s \phi )  \rangle =0$, and therefore
		\begin{equation}\nonumber
			\big|\langle \bar{u}_s \, (\xi_s - \bar{\xi}_s), \phi \rangle \big|	= \big|\langle  \xi_s - \bar{\xi}_s, \bar{u}_s \phi-A(\bar{u}_s \phi ) \rangle \big| \leq \|\xi_s - \bar{\xi}_s\|_{H^{-\tilde{\vartheta}}}  \| \bar{u}_s \phi-A(\bar{u}_s \phi ) \|_{H^{\tilde{\vartheta}}} .
		\end{equation}
		By Poincar\'e's inequality, Lemmas \ref{HHH} and \ref{HCH}, for $ \tilde{\vartheta}\in (1,\frac{3}{2})$ and $\varepsilon \in (0,\frac{1}{2})$, we get
		\begin{equation}\nonumber
			\begin{split}
				\|\bar{u}_s \phi-A(\bar{u}_s \phi ) \|_{H^{\tilde{\vartheta}}} \lesssim \|\nabla(\bar{u}_s \phi)\|_{H^{\tilde{\vartheta}-1}} &\leq \|(\nabla \bar{u}_s) \phi\|_{H^{\tilde{\vartheta}-1}} + \|\bar{u}_s \cdot \nabla \phi\|_{H^{\tilde{\vartheta}-1}} \\
				&\lesssim \|\nabla \bar{u}_s\|_{H^{\frac{1}{2}}} \|\phi\|_{H^{\tilde{\vartheta}-\frac{1}{2}}} + \| \bar{u}_s \|_{C^{\tilde{\vartheta}-1+\varepsilon}}  \|\nabla\phi\|_{H^{\tilde{\vartheta}-1}}\\
				&\lesssim  \|\bar{u}_s\|_{H^{2}} \|\phi\|_{H^{\tilde{\vartheta}}}.
			\end{split}
		\end{equation}
		Summarizing these arguments leads to
		\begin{equation}\nonumber
			\|\bar{u}_s \, ( \xi_s - \bar{\xi}_s)\|_{H^{-\tilde{\vartheta}}} \lesssim  \|\xi_s - \bar{\xi}_s\|_{H^{-\tilde{\vartheta}}}  \|\bar{\xi}_s\|_{H^{1}} .
		\end{equation}
		Inserting the above estimate to \eqref{quantitative lemma 2}, we obtain
		\begin{equation}\label{nonlinear 2nd term}
			\Big\|\int_{0}^{t} e^{(\kappa+\nu)(t-s)\Delta} \big[\bar{u}_s \cdot \nabla( \xi_s -\bar{\xi}_s) \big] ds\Big\|^2_{H^{-\tilde{\vartheta}}} \lesssim \frac1{\kappa+\nu} \int_{0}^{t}  \|\xi_s - \bar{\xi}_s\|^2_{H^{-\tilde{\vartheta}}}  \|\bar{\xi}_s\|^2_{H^{1}} \, ds.
		\end{equation}
		
		Combining \eqref{nonlinear 1st term} and \eqref{nonlinear 2nd term}, by \eqref{nonlinear decompose} we get
		\begin{equation}\label{^2 result}
			\begin{split}
				&\quad \ \Big\|\mathcal{G}\Big(\int_{0}^{\cdot}(u_s\cdot\nabla \xi_s -\bar{u}_s \cdot \nabla \bar{\xi}_s ) \, ds \Big)(t)\Big\|^2_{H^{-\tilde{\vartheta}}} \\
				&\lesssim \frac1{\kappa+\nu} \int_{0}^{t} \|\xi_s-\bar{\xi}_s\|^2_{H^{-\tilde{\vartheta}}} \big(  \|\xi_0\|_{L^{2}}  \|\xi_s\|_{H^{1}} + \|\bar{\xi}_s\|^2_{H^{1}} \big) \, ds .
			\end{split}
		\end{equation}
		According to \eqref{linear decompose}, for $\tilde{\vartheta} \in (1,\frac{3}{2})$, we have
		\begin{equation}\nonumber
			\|\xi_t-\bar{\xi}_t\|^2_{H^{-\tilde{\vartheta}}}\lesssim \|\mathcal{G}(f_t)\|^2_{H^{-\tilde{\vartheta}}}  + \frac1{\kappa+\nu} \int_{0}^{t} \|\xi_s-\bar{\xi}_s\|^2_{H^{-\tilde{\vartheta}}} \big(  \|\xi_0\|_{L^{2}}  \|\xi_s\|_{H^{1}} + \|\bar{\xi}_s\|^2_{H^{1}} \big) \, ds.
		\end{equation}
		By Grönwall's inequality, it holds
		\begin{equation}\label{difference sup}
			\sup_{t\in [0,T]}\|\xi_t-\bar{\xi}_t\|^2_{H^{-\tilde{\vartheta}}} \lesssim \Big( \sup_{t\in [0,T]} \|\mathcal{G}(f_t)\|^2_{H^{-\tilde{\vartheta}}} \Big) \exp\bigg(  \frac{1}{\kappa+\nu} \int_{0}^{T} \big(\|\xi_0\|_{L^{2}} \|\xi_s\|_{H^{1}}+\|\bar{\xi}_s\|^2_{H^{1}} \big) \, ds\bigg).
		\end{equation}
		Notice that $\xi$ and $\bar{\xi}$ satisfy \eqref{priori estimates 1} and \eqref{priori estimates 2} respectively, then H\"older's inequality yields
		\begin{equation}\nonumber
			\begin{split}
				\int_{0}^{T} \big(\|\xi_0\|_{L^{2}}  \|\xi_s\|_{H^{1}}+\|\bar{\xi}_s\|^2_{H^{1}} \big) \, ds
				&\lesssim T^{\frac{1}{2}}\|\xi_0\|_{L^{2}} \Big(\int_{0}^{T} \|\xi_s\|^2_{H^{1}} \, ds \Big)^{\frac{1}{2}} + \int_{0}^{T} \|\bar{\xi}_s\|^2_{H^{1}} \, ds\\
				&\lesssim \kappa^{-\frac{1}{2}} T^{\frac{1}{2}} \|\xi_0\|^2_{L^{2}}  + (\kappa+\nu)^{-1}  \|\xi_0\|^2_{L^{2}} .
			\end{split}
		\end{equation}
		Substituting this estimate into \eqref{difference sup} and setting
		$$C=\frac{\kappa^{-\frac{1}{2}} T^{\frac{1}{2}}+(\kappa+\nu)^{-1}}{\kappa+\nu}, $$
		we arrive at
		\begin{equation*}
			\sup_{t\in [0,T]}\|\xi_t-\bar{\xi}_t\|^2_{H^{-\tilde{\vartheta}}} \lesssim \Big( \sup_{t\in [0,T]} \|\mathcal{G}(f_t)\|^2_{H^{-\tilde{\vartheta}}}  \Big) \exp\big(C \|\xi_0\|^2_{L^{2}}  \big).
		\end{equation*}
		Furthermore, by \eqref{prf of thm 1 1st conclusion}, we have
		\begin{equation}\label{xi-barxi}
			\|\xi-\bar{\xi}\|_{C([0,T],H^{-\tilde{\vartheta}})} \lesssim \|f\|_{C^{\rho\zeta}([0,T],H^{-\tilde{\vartheta}})} \, \exp\Big(\frac{C}{2} \,  \|\xi_0\|^2_{L^{2}}  \Big).
		\end{equation}
		Taking expectation and then applying \eqref{thm1 interpolation}, for $\tilde{\vartheta} \in (1,\frac{3}{2})$, Lemmas \ref{estimate on f} and \ref{sup f H-1} yield
		\begin{equation}\label{xi-barxi expectation}
			\begin{split}
				&\quad \ \mathbb{E} \Big[\|\xi-\bar{\xi}\|_{C([0,T],H^{-\tilde{\vartheta}})}\Big]\\
				&\lesssim \mathbb{E} \Big[ \|f\|_{C^{\rho}([0,T],H^{-\beta})} \Big]^{\zeta } \, \Big(\sup_{t\in [0,T]} \|f_t\|_{H^{-1}}\Big)^{1-\zeta }  \exp\Big(\frac{C}{2} \,  \|\xi_0\|^2_{L^{2}}  \Big)\\
				&\lesssim T^{1+\frac{\gamma}{2}(\zeta-1)} \|\xi_0\|_{L^2}  \big(\nu^{1+\frac{\gamma}{2}}\alpha^{-\epsilon}  +\nu^{\frac{1}{2}}\|\theta\|_{\ell^{\infty}}\big)^{\zeta}  \big(\nu \kappa^{-\frac{1}{2}}+\kappa^{-\frac{\gamma}{2}}\big)^{1-\zeta} \exp\Big(\big(1+\frac{C}{2}\big) \,  \|\xi_0\|^2_{L^{2}}  \Big),
			\end{split}
		\end{equation}	
		where the last step follows from
		\begin{equation}\nonumber
			\big(1+ \|\xi_0\|_{L^2}\big)^{1+\zeta} <  \big(1+ \|\xi_0\|_{L^2}\big)^{2} \leq 2\,  \big(1+ \|\xi_0\|^2_{L^2}\big) \leq 2 \exp(\|\xi_0\|^2_{L^2}), \quad \zeta \in (0,1).
		\end{equation}
	If we take $C_1 \sim T^{1+\frac{\gamma}{2}(\zeta-1)} \big(\nu \kappa^{-\frac{1}{2}}+\kappa^{-\frac{\gamma}{2}}\big)^{1-\zeta} $ and $C_2=1+\frac{C}{2}$, then \eqref{xi-barxi expectation} can be rewritten as
		\begin{equation}\label{xi-xibar final}
			\mathbb{E} \Big[\|\xi-\bar{\xi}\|_{C([0,T],H^{-\tilde{\vartheta}})}\Big] \leq C_1  \|\xi_0\|_{L^2}  \exp\big(C_2 \|\xi_0\|^2_{L^2} \big) \big(\nu^{1+\frac{\gamma}{2}}\alpha^{-\epsilon}  +\nu^{\frac{1}{2}}\|\theta\|_{\ell^{\infty}}\big)^{\zeta}.
		\end{equation}

		Once we have the estimate for $\tilde{\vartheta}>1$, then for $\hat{\vartheta} \in (0,1)$, we deduce from Lemma \ref{interpolation} that
		\begin{equation}\label{vartheta in (0,1)}
			\begin{split}
				\mathbb{E} \Big[\|\xi-\bar{\xi}\|_{C([0,T],H^{-\hat{\vartheta}})}\Big] &\leq \mathbb{E} \Big[\|\xi-\bar{\xi}\|^{\hat{\vartheta}/\tilde{\vartheta}}_{C([0,T],H^{-\tilde{\vartheta}})}\Big] \, \|\xi_0\|_{L^2}^{1-\hat{\vartheta}/\tilde{\vartheta}}\\
				&\leq C_1^{\hat{\vartheta} /\tilde{\vartheta}} \|\xi_0\|_{L^2} \exp \big(C_2 (\hat{\vartheta}/\tilde{\vartheta}) \|\xi_0\|^2_{L^{2}}\big) \,   \big(\nu^{1+\frac{\gamma}{2}}\alpha^{-\epsilon} +\nu^{\frac{1}{2}}\|\theta\|_{\ell^{\infty}}\big)^ {\zeta  \hat{\vartheta} /\tilde{\vartheta}} \\
				&\leq  C_1  \|\xi_0\|_{L^2}  \exp\big(C_2 \|\xi_0\|^2_{L^2} \big) \big(\nu^{1+\frac{\gamma}{2}}\alpha^{-\epsilon}  +\nu^{\frac{1}{2}}\|\theta\|_{\ell^{\infty}}\big)^{\zeta} ,
			\end{split}
		\end{equation}	
		where in the last step we have used the fact that $C_1>1$, $C_2>0$ and $\hat{\vartheta}/\tilde{\vartheta}<1$. Hence it holds $C_1^{\hat{\vartheta} /\tilde{\vartheta}}<C_1$ and $C_2 \, \hat{\vartheta}/\tilde{\vartheta} < C_2$. Combining \eqref{xi-xibar final} and  \eqref{vartheta in (0,1)}, we obtain the final conclusion of Theorem \ref{thm1} for every $\vartheta>0$.
	\end{proof}

	\subsection{Proof of Theorem \ref{thm2}}\label{subsec-proof-thm-2}

	In order to prove Theorem \ref{thm2}, we first present several simple results as follows.
	\begin{lemma}\label{energy estimate}
		For the solution of \eqref{SPDE}, the energy equality holds with probability one:
		$$	\frac{d}{dt}\,  \|\xi\|^2_{L^2} = -2\kappa \, \|\xi\|^2_{H^1}. $$
	\end{lemma}
	
	\begin{proof}
		This result is well-known; we present the proof for completeness. Notice that
		\begin{equation}\nonumber
			\frac{d}{dt}\, \|\xi\|^2_{L^2} = \frac{d}{dt} \,\langle \xi, \xi \rangle = \int_{\mathbb{T}^2} 2 \xi \, \frac{\partial \xi}{\partial t} \, dx.
		\end{equation}
		By \eqref{SPDE} , we can further get
		\begin{equation*}
			\begin{split}
				\int_{\mathbb{T}^2} 2 \xi \, \frac{\partial \xi}{\partial t} \, dx &= 2\kappa \int_{\mathbb{T}^2} \xi \, \Delta \xi \, dx - 2 \int_{\mathbb{T}^2} \xi  u \cdot \nabla \xi \, dx -2\int_{\mathbb{T}^2}\xi \,\bm{b} \cdot \nabla \xi \, dx\\
				&=-2\kappa \int_{\mathbb{T}^2} |\nabla \xi|^2 \, dx + \int_{\mathbb{T}^2} {\rm div}(u) \, \xi^2 \, dx + \int_{\mathbb{T}^2} {\rm div}(\bm{b}) \, \xi^2 \,dx\\
				&=-2\kappa \, \|\xi\|^2_{H^1}.\quad \qedhere
			\end{split}
		\end{equation*}
	\end{proof}
	
Recall that $\{\mathcal{F}_t\}_{t\ge 0}$ is the filtration on the probability space $\Omega$. The following estimate is an easy consequence of Theorem \ref{thm1}.

	\begin{lemma}\label{thm1 corollary}
		For any $n \in \mathbb{N}$, let $\{\bar{\xi}^n_t\}_{t\ge n}$ be the solution to
		\begin{equation}\label{deterministic pde with initial time n}
			\left\{ \aligned
			& \partial_t \bar\xi^n+ \bar u^n\cdot\nabla \bar\xi^n =(\kappa+\nu) \Delta \bar\xi^n, \quad t\geq n,\\
			& \bar u^n= K\ast \bar\xi^n, \quad \bar\xi^n|_{t=n} = \xi_n.
			\endaligned
			\right.
		\end{equation}
Then it holds, $\P$-a.s.,
		$$\E \Big[\sup_{t\in [n,n+1]} \|\xi_t-\bar{\xi}^n_{t}\|_{H^{-1}} \big| \mathcal{F}_n\Big] \leq C_1  \|\xi_n\|_{L^2}  \exp\big(C_2 \|\xi_0\|^2_{L^2} \big) \big(\nu^{1+\frac{\gamma}{2}}\alpha^{-\epsilon}  +\nu^{\frac{1}{2}}\|\theta\|_{\ell^{\infty}}\big)^{\zeta} ,$$
		where $C_1$ and $C_2$ are defined as in Theorem \ref{thm1} and are independent of $n$.
	\end{lemma}

\begin{proof}
Notice that if we take $\vartheta=1$ and $T=1$ in Theorem \ref{thm1}, then we get a quantitative estimate on the distance between the solutions of \eqref{SPDE} and \eqref{PDE}, both with the same deterministic initial value $\xi_0$. Since the Ornstein-Uhlenbeck flow $\bm{b}$ in \eqref{SPDE} is a stationary process, such estimate holds on any unit interval of the form $[n,n+1]$, as long as we restart \eqref{PDE} at the time $t=n$ with the same value $\xi_n$. However, as $\xi_n$ is random, we need to take conditional expectation with respect to $\mathcal{F}_n$ and get the desired result.
\end{proof}
	
	\begin{lemma}\label{decay of L2 norm}
		For all $n \in \mathbb{N}$, decay of $L^2$-norm of the solution to \eqref{deterministic pde with initial time n} satisfies	
		$$	\|\bar{\xi}_{t}^n\|^2_{L^2}\leq e^{-\lambda_1(t-n)} \, \|\xi_n\|^2_{L^2}, $$
		where $\lambda_1:=8\pi^2 (\kappa+\nu)$ is the principal eigenvalue of $(\kappa+\nu)\Delta$ on $\T^2$.
	\end{lemma}
	
	\begin{proof}
		By \eqref{deterministic pde with initial time n}, we use similar method as Lemma \ref{energy estimate} to get
		\begin{equation}\nonumber
			\frac{d}{dt}\, \|\bar{\xi}^n\|^2_{L^2}= -2(\kappa+\nu) \, \|\nabla \bar{\xi}^n \|^2_{L^2}.
		\end{equation}
		Poincar\'e's inequality yields $\|\bar{\xi}^n\|^2_{L^2} \leq \frac{1}{4\pi^2} \, \|\nabla\bar{\xi}^n\|^2_{L^2}$,
		thus
		\begin{equation}\nonumber
			\frac{d}{dt}\, \|\bar{\xi}^n\|^2_{L^2} \leq -8\pi^2 (\kappa+\nu) \|\bar{\xi}^n\|^2_{L^2}.
		\end{equation}
		  Solving the differential inequality gives us the desired estimate.
	\end{proof}

\begin{remark}
		In particular, if we consider \eqref{deterministic pde with initial time n} with initial time $n=0$, then it reduces to equation \eqref{PDE}, and we get the decay rate of $L^2$-norm for the solution to \eqref{PDE} as
	$$	\|\bar{\xi}_{t}\|^2_{L^2}\leq e^{-\lambda_1 t} \, \|\xi_0\|^2_{L^2}. $$	
\end{remark}

	On the basis of the above results, now we can provide
	
	\begin{proof}[Proof of Theorem \ref{thm2}] Since the proof is rather long, we divide it into the following four steps.
	
	\emph{Step 1.}
	Let $R>0$ be given as in the statement of Theorem \ref{thm2}, and denote
  $$c_1:= C_1^{1/2} \exp\Big(\frac{C_2 R^2}{2}  \Big) \big(\nu^{1+\frac{\gamma}{2}} \alpha^{-\epsilon} +\nu^{\frac{1}{2}} \|\theta\|_{\ell^{\infty}}\big)^{\zeta / 2},$$
which is sufficiently small by taking $\alpha$ big and $\|\theta\|_{\ell^{\infty}}$ small. Recall that we have assumed $\|\xi_0\|_{L^2} \le R$; then by Lemma \ref{thm1 corollary}, we have, $\P$-a.s.,
  \begin{equation}\label{proof-thm-1.1}
  \E \Big[\sup_{t\in [n,n+1]} \|\xi_t-\bar{\xi}^n_{t}\|_{H^{-1}} \big| \mathcal{F}_n\Big] \leq c_1^2 \|\xi_n\|_{L^2}.
  \end{equation}
Define the event
  $$A_n:= \Big\{\omega \in \Omega: \sup_{t\in [n,n+1]} \|\xi_t(\omega) -\bar{\xi}^n_{t}(\omega)\|_{H^{-1}}> c_1\|\xi_n(\omega)\|_{L^2}\Big\},$$
and $A_n^c$ is its complement; we want to prove
	\begin{equation}\label{prf of thm2 step1}
		\P(A_n^c) > 1-c_1,
	\end{equation}
which is an easy consequence of
	\begin{equation}\nonumber
		\P \big(A_n |\mathcal{F}_n \big)= \E \big[{\bf 1}_{A_n}|\mathcal{F}_n\big] \leq c_1.
	\end{equation}
Indeed, for any $B \in \mathcal{F}_n$, it holds
$$\aligned
\int_B \E \big[{\bf 1}_{A_n}|\mathcal{F}_n\big] \, d\P= \int_B {\bf 1}_{A_n} \, d\P  &\leq \int_{A_n \cap B}\frac{\sup_{t\in [n,n+1]} \|\xi_t -\bar{\xi}^n_{t}\|_{H^{-1}}}{c_1 \|\xi_n\|_{L^2}} \, d\P\\
&\leq c_1^{-1} \int_B \|\xi_n\|_{L^2}^{-1} \sup_{t\in [n,n+1]} \|\xi_t -\bar{\xi}^n_{t}\|_{H^{-1}} \, d\P\\
&=c_1^{-1} \int_B \E \Big[\|\xi_n\|_{L^2}^{-1} \sup_{t\in [n,n+1]} \|\xi_t -\bar{\xi}^n_{t}\|_{H^{-1}} \big| \mathcal{F}_n \Big]\, d\P.
\endaligned $$
Then by the arbitrariness of $B$ and \eqref{proof-thm-1.1}, we obtain
\begin{equation}\label{conditional expectation of IAn}
	\begin{split}
		\E \big[{\bf 1}_{A_n}|\mathcal{F}_n\big] &\leq c_1^{-1}  \E \Big[\|\xi_n\|_{L^2}^{-1} \sup_{t\in [n,n+1]} \|\xi_t -\bar{\xi}^n_{t}\|_{H^{-1}} \big| \mathcal{F}_n \Big]\\
		&=c_1^{-1} \|\xi_n\|_{L^2}^{-1} \, \E  \Big[ \sup_{t\in [n,n+1]} \|\xi_t -\bar{\xi}^n_{t}\|_{H^{-1}} \big| \mathcal{F}_n \Big]\\
		&\leq c_1^{-1} \|\xi_n\|_{L^2}^{-1} \, c_1^2 \|\xi_n\|_{L^2}=c_1.
	\end{split}
\end{equation}

\emph{Step 2.} In order to estimate the decay rate of $L^2$-norm for $\xi$, we first try to find the relationship between $\|\xi_n\|_{L^2}$ and $\|\xi_{n+1}\|_{L^2}$ for any $n\in \mathbb{N}$, then we use iteration to extend the conclusion to $\|\xi_t\|_{L^2}$ with initial value $\xi_0$ for all $ t \geq 0$.

Notice that for $t \in [n,n+1]$, Lemma \ref{decay of L2 norm} and inequality \eqref{prf of thm2 step1} yield, with probability no less than $1-c_1$,
\begin{equation}\nonumber
	\|\xi_{t}\|^2_{H^{-1}} \leq 2 \|\bar{\xi}^n_{t}\|^2_{H^{-1}} +2 \|\xi_t-\bar{\xi}^n_{t}\|^2_{H^{-1}} \leq 2 \|\xi_n\|^2_{L^2}\big(e^{-\lambda_1 (t-n)}+c_1^2\big).
\end{equation}
Besides, according to Lemma \ref{interpolation}, we have $\|\xi_{t}\|^2_{L^2} \leq \|\xi_{t}\|_{H^{-1}} \|\xi_t\|_{H^1}$. Hence, combining Lemma \ref{energy estimate} with the above two inequalities leads to
\begin{equation}\nonumber
	\frac{d}{dt}\, \|\xi_t\|^2_{L^2}  \leq -2\kappa \frac{\|\xi_t\|^4_{L^2}}{\|\xi_{t}\|^2_{H^{-1}}} \leq -\frac{\kappa \|\xi_t\|^4_{L^2}}{\|\xi_n\|^2_{L^2}\big(e^{-\lambda_1 (t-n)} +c_1^2\big)}, \quad t\in[n,n+1].
\end{equation}
Solving the differential inequality and then letting $t=n+1$, we get, on the event $A_n^c$,
	\begin{equation}\label{decay of xi}
		\|\xi_{n+1}\|^2_{L^2} \leq \frac{\|\xi_n\|^2_{L^2}}{1+\frac{\kappa}{\lambda_1 c_1^2 }\log\frac{1+c_1^2 e^{\lambda_1 }}{1+c_1^2}}=:c_2^2 \|\xi_n\|_{L^2}^2.
	\end{equation}
As the $L^2$-norm of $\xi$ is decreasing, we additionally apply \eqref{decay of xi} to further get
$$\aligned
\E \|\xi_{n+1}\|_{L^2} &=\E \big[\|\xi_{n+1}\|_{L^2} \, {\bf 1}_{A_n}\big]+\E \big[\|\xi_{n+1}\|_{L^2}\, {\bf 1}_{A^c_n}\big] \le \E \big[\|\xi_{n}\|_{L^2} \, {\bf 1}_{A_n}\big] + \E \big[c_2\|\xi_{n}\|_{L^2} \big].\\
\endaligned $$
Using the property of conditional expectation, \eqref{conditional expectation of IAn}  yields
$$\aligned
\E \|\xi_{n+1}\|_{L^2} &\leq \E \Big[\E \big[\|\xi_{n}\|_{L^2} \, {\bf 1}_{A_n}\big| \mathcal{F}_n\big] \Big] +c_2\, \E \|\xi_{n}\|_{L^2}\\
&=\E \Big[\|\xi_n\|_{L^2} \, \E \big[{\bf 1}_{A_n} \big| \mathcal{F}_n\big]\Big]+c_2\, \E \|\xi_n\|_{L^2}\\
&\leq (c_1+c_2)\, \E \|\xi_n\|_{L^2}.
\endaligned $$
Afterwards, we denote $c_0:=c_1+c_2$ for simplicity. By induction, for any $n \in \mathbb{N}$, it holds
\begin{equation}\label{iteration}
	\E \|\xi_n\|_{L^2} \leq c_0^n \|\xi_0\|_{L^2}.
\end{equation}

\emph{Step 3.} To show the enhanced dissipation property of Ornstein-Uhlenbeck flow, the constant $c_0>0$ has to be sufficiently small. We start with proving the following quantity in the definition of $c_2$ can be very large under suitable choice of parameters:
$$\frac{\kappa}{\lambda_1 c_1^2 }\log\frac{1+c_1^2 e^{\lambda_1}}{1+c_1^2}= \frac{\kappa}{\lambda_1 c_1^2} \log \bigg(\frac{c_1^2}{1+c_1^2}\big(e^{\lambda_1}-1\big)+1\bigg).$$
First, fix a sufficiently large $\nu$, then $\lambda_1=8\pi^2(\kappa+\nu)$ is also very large. Next, we let $\alpha$ be large and $\|\theta\|_{\ell^\infty}$ be small enough, and thus $c_1$ is sufficiently small by its definition. In particular, we can assume that $\frac{c_1^2}{1+c_1^2}(e^{\lambda_1}-1) \in (0,1]$. As $\log(1+x) \geq x\log2$ for $x\in (0,1]$, we have
$$\frac{\kappa}{\lambda_1 c_1^2 } \log \Big(\frac{c_1^2}{1+c_1^2}\big(e^{\lambda_1}-1\big)+1\Big) \geq \frac{\kappa\log2}{\lambda_1 c_1^2 }  \cdot \frac{c_1^2}{1+c_1^2}\big(e^{\lambda_1}-1\big) \geq \frac{\kappa\log2}{2} \cdot \frac{e^{\lambda_1} -1}{\lambda_1},$$
where in the last step we have used the fact that $c_1^2+1 \leq 2$. Since $\lambda_1$ is large, we deduce that the left-hand side is also very large; as a result, $c_2$ can be very small. Combined with the smallness of $c_1$, we conclude that $c_0= c_1 +c_2$ is also a small constant.

\emph{Step 4.} Based on the previous discussions, we can now prove the final conclusion of Theorem \ref{thm2}. Define $\lambda_0:=-\log c_0>0$, which can be assumed to be greater than $\lambda(1+p)$, where $\lambda$ and $p$ are given in the statement of Theorem \ref{thm2}. Then by \eqref{iteration}, it holds
$$\E \Big[\sup_{t\in [n,n+1]} \|\xi_t\|_{L^2}\Big]=\E \|\xi_n\|_{L^2} \leq e^{-\lambda_0 n} \|\xi_0\|_{L^2}. $$
We define the events
$$E_n:=\Big\{\omega \in \Omega: \sup_{t\in [n,n+1]} \|\xi_t(\omega)\|_{L^2}> e^{-\lambda n} \|\xi_0\|_{L^2}\Big\}, \quad n\in \mathbb{N}.$$
By Markov's inequality, it holds
$$\sum_{n\in\mathbb{N}} \P(E_n) \leq \sum_{n\in\mathbb{N}} \frac{e^{\lambda n}}{\|\xi_0\|_{L^2}}\E \Big[\sup_{t\in [n,n+1]} \|\xi_t\|_{L^2}\Big] \leq \sum_{n\in\mathbb{N}} e^{(\lambda-\lambda_0)n} < +\infty.$$
Furthermore, Borel-Cantelli's lemma implies that for $\P$-a.s. $\omega \in \Omega$, there exists $N(\omega) \in \mathbb{N}$, such that for any $n>N(\omega)$,
$$\sup_{t\in [n,n+1]} \|\xi_t\|_{L^2} \leq e^{-\lambda n} \|\xi_0\|_{L^2}.$$
For the case $0 \leq n \leq N(\omega)$, we have
$$\sup_{t\in [n,n+1]} \|\xi_t\|_{L^2} =\|\xi_n\|_{L^2}=e^{\lambda n} e^{-\lambda n} \|\xi_n\|_{L^2} \leq e^{\lambda N(\omega)} e^{-\lambda n} \|\xi_0\|_{L^2}.$$
If we let $C(\omega)=e^{\lambda (1+N(\omega))}$, it is not difficult to verify that for $\P$-a.s. $\omega\in \Omega$,
\begin{equation}\label{thm2 conclusion}
	\|\xi_t\|_{L^2} \leq C(\omega) e^{-\lambda t} \|\xi_0\|_{L^2}, \quad \forall\, t\geq 0.
\end{equation}

As for the finite $p$-th moment of $C(\omega)$, a similar proof can be found in \cite[Section 5.2]{FLD quantitative}, so we omit it here.
\end{proof}

	\section{Proof of Proposition \ref{main proposition}}
	
	We devote this section to the proof of Proposition \ref{main proposition}. We first divide the desired quantity into a summation part and two integrals. For the summation term, we follow the idea of \cite{Pappalettera} and use equation \eqref{SPDE} to further decompose it. Then we will estimate each of the decomposed terms and the two integrals separately. Finally,  in order to obtain the desired estimate, we need to make some restrictions on the parameters.
	
	\subsection{Decomposition}\label{subsec-decomp}

	We decompose the quantity we want to estimate as follows:
	\begin{equation}\label{decomposition of fn-fm}
		\begin{split}
			&\quad \  \xi_{n\delta}-\xi_{m\delta}-(\kappa+\nu)\int_{m\delta}^{n\delta} \Delta \xi_s \, ds+\int_{m\delta}^{n\delta} u_s\cdot \nabla \xi_s \, ds\\
			&=\xi_{n\delta}-\xi_{m\delta}-\delta(\kappa+\nu)\sum_{h=m}^{n-1} \Delta \xi_{h\delta}+ \delta \sum_{h=m}^{n-1} (u_{h\delta} \cdot \nabla \xi_{h\delta})+I_a+I_b,
		\end{split}
	\end{equation}
	where
	$$I_{a}:=(\kappa+\nu) \int_{m\delta}^{n\delta} \Delta\big( \xi_{[s]}-\xi_s\big) ds , \quad I_b:= \int_{m\delta}^{n\delta} \big(u_s \cdot \nabla \xi_s-u_{[s]} \cdot \nabla \xi_{[s]}\big) ds,$$
	and $[s]:=\sup_{j\in \mathbb{N}}\{j\delta:j\delta \leq s \}$. We first consider
	$$\xi_{n\delta}-\xi_{m\delta}-\delta(\kappa+\nu)\sum_{h=m}^{n-1} \Delta \xi_{h\delta}+ \delta \sum_{h=m}^{n-1} (u_{h\delta} \cdot \nabla \xi_{h\delta})  ,  \quad n>m,	$$
	which can also be written as
	$$	\xi_{(n+1)\delta}-\xi_{m\delta}-\delta(\kappa+\nu)\sum_{h=m}^{n} \Delta\xi_{h\delta}+\delta \sum_{h=m}^{n} (u_{h\delta} \cdot \nabla \xi_{h\delta}) , \quad  n\geq m.	 $$
	
	According to Definition \ref{solution of 1.1},  for every $h=0,1,\ldots,T/\delta-1$, we have
	\begin{equation} \label{decompose 1}
		\begin{split}
			\xi_{(h+1)\delta}-\xi_{h\delta}&=-\int_{h\delta}^{(h+1)\delta} u_s \cdot \nabla \xi_s \, ds- \int_{h\delta}^{(h+1)\delta} \bm{b}(s) \cdot \nabla \xi_s \, ds+ \kappa \int_{h\delta}^{(h+1)\delta} \Delta \xi_s  \,ds\\
			&=:I_1(h)+I_2(h)+I_3(h).
		\end{split}
	\end{equation}
	Furthermore, we can make the following decomposition for $I_2(h)$:
	\begin{equation}\nonumber
		\begin{split}
			I_2(h)&=-\int_{h\delta}^{(h+1)\delta} \bm{b}(s) \cdot \nabla ( \xi_s-\xi_{h\delta} ) \, ds-\int_{h\delta}^{(h+1)\delta} \bm{b}(s) \cdot \nabla \xi_{h\delta} \, ds\\
			&=:I_{21}(h)+I_{22}(h)+I_{23}(h)+I_{24}(h)+I_{25}(h),
		\end{split}
	\end{equation}
	where $I_{2i}(h), i=1,\ldots,5$ are defined as
	\begin{align*}
		I_{21}(h)&=\int_{h\delta}^{(h+1)\delta}\int_{h\delta}^{s}\bm{b}(s) \cdot \nabla(u_r \cdot \nabla \xi_r) \, drds,\\
		I_{22}(h)&=\int_{h\delta}^{(h+1)\delta}\int_{h\delta}^{s}\bm{b}(s) \cdot \nabla \big(\bm{b}(r)\cdot \nabla (\xi_r-\xi_{h\delta})\big) \, drds,\\
		I_{23}(h)&= \int_{h\delta}^{(h+1)\delta}\int_{h\delta}^{s} \bm{b}(s) \cdot \nabla \big(\bm{b}(r) \cdot \nabla \xi_{h\delta}\big)\, drds,\\
		I_{24}(h)&=-\kappa \int_{h\delta}^{(h+1)\delta}\int_{h\delta}^{s}\bm{b}(s) \cdot \nabla (\Delta \xi_r)\, drds,\\
		I_{25}(h)&=-\int_{h\delta}^{(h+1)\delta}\bm{b}(s) \cdot \nabla \xi_{h\delta} \, ds.
	\end{align*}
	By the definition of $\bm{b}$, the term $I_{23}(h)$ can be rewritten as follows:
	\begin{equation}\nonumber
		\begin{split}
			I_{23}(h)&=4\nu \sum_{k,k'\in \mathbb{Z}_0^2}\int_{h\delta}^{(h+1)\delta}\int_{h\delta}^{s} \theta_{k}\sigma_{k}\eta^{\alpha,{k}}(s) \cdot \nabla( \theta_{k'}\sigma_{k'}\eta^{\alpha,k'}(r) \cdot \nabla \xi_{h\delta}) \, drds\\
			&=4\nu \sum_{k,k'\in \mathbb{Z}_0^2} \theta_{k}\sigma_{k} \cdot \nabla( \theta_{k'}\sigma_{k'} \cdot \nabla \xi_{h\delta})  \bigg( \int_{h\delta}^{(h+1)\delta}\int_{h\delta}^{s} \eta^{\alpha,k}(s)  \eta^{\alpha,{k'}}(r)  \, drds-   \delta_{k,k'} \frac{\delta}{2}\bigg)\\
			&\quad+ 2\nu \delta \sum_{k\in \mathbb{Z}_0^2} \theta_{k}\sigma_{k} \cdot \nabla(\theta_{k}\sigma_{k} \cdot \nabla \xi_{h\delta} )\\
			&=:I_{231}(h)+I_{232}(h).
		\end{split}
	\end{equation}	
	Using the radial symmetry of $\theta\in \ell^2(\Z^2_0)$ and the expression of $\sigma_k$, it is not difficult to prove the following identity (see e.g. \cite[Lemma 2.1]{FLarxiv})
	\begin{equation*}
		\sum_{k\in \Z_0^2}\theta_k^2 (\sigma_k \otimes  \sigma_k )(x)= \frac{1}{2} \|\theta\|_{\ell^2}^2 \, Id,
	\end{equation*}
	where $Id$ is the two dimensional identity matrix; hence by ${\rm div} \sigma_{k}=0$ and $\|\theta\|_{\ell^2}^2=1$, we have
	\begin{equation}\label{citation 1}
		\begin{split}
			\sum_{k\in \mathbb{Z}_0^2}\theta_{k}\sigma_{k} \cdot \nabla(\theta_{k}\sigma_{k} \cdot \nabla \xi_{h\delta} ) ={\rm div}\Big(\sum_{k \in \mathbb{Z}_0^2} \theta_k^2 \, (\sigma_{k} \otimes \sigma_{k}) \nabla \xi_{h\delta} \Big)= \frac{1}{2}\Delta\xi_{h\delta} .
		\end{split}
	\end{equation}
	Therefore, $I_{232}(h)=\nu \delta \Delta \xi_{h\delta}$.
	As for the term $I_3(h)$, we can divide it into two terms:
	\begin{equation}\nonumber
		I_{3}(h)=\kappa \int_{h\delta}^{(h+1)\delta} \Delta(\xi_s - \xi_{h\delta}) \, ds + \kappa \delta  \Delta \xi_{h\delta}  =:I_{31}(h)+I_{32}(h).
	\end{equation}
	
	Taking the sum of \eqref{decompose 1} over $h=m,\ldots,n$, and noticing that $$I_{232}(h)+I_{32}(h)=\delta(\kappa+\nu) \Delta \xi_{h\delta} ,$$
	we can finally get
	\begin{equation}\label{decompose 2}
		\begin{split}
			&\quad \ \xi_{(n+1)\delta}-\xi_{m\delta}-\delta(\kappa+\nu)\sum_{h=m}^{n} \Delta\xi_{h\delta}+\delta \sum_{h=m}^{n} (u_{h\delta} \cdot \nabla \xi_{h\delta}) \\
			&=\sum_{h=m}^{n} \big(I_{1}(h)+I_{21}(h)+I_{22}(h)+I_{231}(h)+I_{24}(h)+I_{25}(h)+I_{31}(h)+\delta  (u_{h\delta} \cdot \nabla  \xi_{h\delta})  \big) .
		\end{split}
	\end{equation}

In the following several subsections, we will estimate each term of the above formula. For readers' convenience, we give a brief introduction here. The estimates on terms $I_{21}(h)$, $I_{22}(h)$, $I_{24}(h)$ and $I_{31}(h)$ will be given in Section \ref{subs-5.2}. As the terms $I_{231}(h)$ and $I_{25}(h)$ are more technical to deal with, we consider them in Sections \ref{subsec-231} and \ref{subsec-25}, respectively. In Section \ref{subsec-1-drift}, we treat $I_1(h)$ together with $\delta(u_{h\delta} \cdot \nabla  \xi_{h\delta})$. We mention that the two remaining intergrals $I_a$ and $I_b$ in \eqref{decomposition of fn-fm} are similar to $I_{31}(h)$ and $I_1(h)+\delta(u_{h\delta}\cdot \nabla \xi_{h\delta})$, respectively, hence we give their estimates in Section \ref{subsec-a-b} without proof. Finally, we combine all the estimates and provide in Section \ref{subs-proof-Prop-3.3} the proof of Proposition \ref{main proposition}.

	\subsection{The terms $I_{21}(h),I_{22}(h),I_{24}(h),I_{31}(h)$}\label{subs-5.2}
	
	The estimates on these four terms are collected in the next lemma.
	
	\begin{lemma}\label{prop3 1st lemma}
		Let $\gamma \in (0,\frac{1}{3})$, $\beta>3$ and $T\geq 1$, then we have the following estimates:
		$$	\aligned \mathbb{E} \bigg[\sup_{1\leq m<n \leq T/\delta-1}\Big\|\sum_{h=m}^{n}I_{21}(h)\Big\|_{H^{-\beta}}\bigg] &\lesssim  \delta \nu^{\frac{1}{2}} \alpha^{\frac{1}{2}} \, TC^{1/2}_{\theta,2-\gamma,2} \,  \|\xi_{0}\|^2_{L^2}   ,\\
		\mathbb{E} \bigg[\sup_{1\leq m<n \leq T/\delta-1}\Big\|\sum_{h=m}^{n}I_{22}(h)\Big\|_{H^{-\beta}}\bigg]  &\lesssim  \delta^{1+\gamma} \nu^{1+\frac{\gamma}{2}} \alpha^{1+\frac{\gamma}{2}}  \, T C^{\frac{1}{2}+\frac{\gamma}{4}}_{\theta,2-\gamma,4}  \, \|\xi_0\|_{L^2} \big(1+\|\xi_0\|_{L^2} \big)^{\gamma}  ,\\
		\mathbb{E} \bigg[\sup_{1\leq m<n \leq T/\delta-1}\Big\|\sum_{h=m}^{n}I_{24}(h)\Big\|_{H^{-\beta}}\bigg] &\lesssim \kappa^{\frac{1+\gamma}{2}} \delta^{\frac{1+\gamma}{2}} \nu^{\frac{1}{2}} \alpha^{\frac{1}{2}} \, T C^{1/2}_{\theta,2-\gamma,2} \,  \|\xi_{0}\|_{L^2}  ,\\
		\mathbb{E} \bigg[\sup_{1\leq m<n \leq T/\delta-1}\Big\|\sum_{h=m}^{n}I_{31}(h)\Big\|_{H^{-\beta}}\bigg] &\lesssim   \kappa \delta^{\gamma} \nu^{\frac{\gamma}{2}} \alpha^{\frac{\gamma}{2}} \, T C^{\gamma/2}_{\theta,1+\gamma,2} \,  \|\xi_{0}\|_{L^2} \big(1+\|\xi_{0}\|_{L^2} \big)^{\gamma}.
		\endaligned $$
	\end{lemma}

	\begin{proof}
		First, we consider the term
		$$	I_{21}(h)=\int_{h\delta}^{(h+1)\delta}\int_{h\delta}^{s}\bm{b}(s) \cdot \nabla(u_r \cdot \nabla \xi_r) \, drds.$$
		Taking a test function $\phi \in H^{\beta}(\T^2)$ and integrating by parts, it holds
		$$\aligned
		\big|\big\langle\bm{b}(s) \cdot \nabla(u_r \cdot \nabla \xi_r) ,\phi \big\rangle\big| & =\big|\big\langle u_r \cdot \nabla(\bm{b}(s) \cdot \nabla \phi) , \xi_r \big\rangle \big|\\
		&\leq \|u_r \cdot \nabla(\bm{b}(s) \cdot \nabla \phi) \|_{L^2} \|\xi_r\|_{L^2} \\
		&\le \|u_r\|_{H^{1-\gamma}} \|\nabla(\bm{b}(s)\cdot\nabla\phi)\|_{H^{\gamma}}\|\xi_0\|_{L^2},
		\endaligned $$
		where the last step follows from Lemma \ref{HHH}; again by Lemma \ref{HHH}, for $\gamma \in (0,\frac{1}{2})$, we have
		\begin{equation}\label{decompose H^g}
			\begin{split}
				\|\nabla(\bm{b}(s)\cdot\nabla\phi)\|_{H^{\gamma}} & \leq \|\nabla \bm{b}(s) \cdot \nabla \phi \|_{H^{\gamma}} +\|\bm{b}(s)\cdot \nabla^2 \phi\|_{H^{\gamma}}\\
				&\lesssim \|\nabla\bm{b}(s)\|_{H^{1-\gamma}} \|\nabla\phi\|_{H^{2\gamma}}+\|\bm{b}(s)\|_{H^{1-\gamma}} \|\nabla^2\phi\|_{H^{2\gamma}}\\
				&\lesssim \|\bm{b}(s)\|_{H^{2-\gamma}} \|\phi\|_{H^{1+2\gamma}} + \|\bm{b}(s)\|_{H^{1-\gamma}}\|\phi\|_{H^{2+2\gamma}}\\
				&\lesssim \|\bm{b}(s)\|_{H^{2-\gamma}} \|\phi\|_{H^{2+2\gamma}}.
			\end{split}
		\end{equation}
		Combining the above two estimates, we can easily get for $\beta>3$,
		$$\big\|\bm{b}(s) \cdot \nabla(u_r \cdot \nabla \xi_r)\big\|_{H^{-\beta}} \lesssim \|u_r\|_{H^{1-\gamma}} \|\bm{b}(s)\|_{H^{2-\gamma}} \|\xi_0\|_{L^2}  \lesssim \|\xi_0\|^2_{L^2} \, \|\bm{b}(s)\|_{H^{2-\gamma}} .$$
		Furthermore,
		\begin{equation}\nonumber
			\begin{split}
				\big\|I_{21}(h)\big\|_{H^{-\beta}}  \lesssim  \|\xi_0\|^2_{L^2}  \int_{h\delta}^{(h+1)\delta} \int_{h\delta}^s \|\bm{b}(s)\|_{H^{2-\gamma}}\, drds \lesssim  \delta \,  \|\xi_0\|^2_{L^{2}}   \int_{h\delta}^{(h+1)\delta} \|\bm{b}(s)\|_{H^{2-\gamma}} \, ds .
			\end{split}
		\end{equation}	
		Taking supremum and then expectation, Lemma \ref{lemma Eb} yields
		\begin{equation}\nonumber
			\begin{split}
				\mathbb{E} \bigg[\sup_{1\leq m<n \leq T/\delta-1}\Big\|\sum_{h=m}^{n}I_{21}(h)\Big\|_{H^{-\beta}}\bigg] &\leq \sum_{h=1}^{T/\delta-1} \mathbb{E}\Big[\big\|I_{21}(h)\big\|_{H^{-\beta}}\Big]\\
				&\lesssim  \delta \, \|\xi_0\|^2_{L^{2}}    \int_{0}^{T} \Big(\mathbb{E}  \big[ \|\bm{b}(s)\|^2_{H^{2-\gamma}} \big]\Big)^{\frac{1}{2}}ds \\
				&\lesssim \delta \nu^{\frac{1}{2}} \alpha^{\frac{1}{2}} \, T C^{1/2}_{\theta,2-\gamma,2} \,  \|\xi_{0}\|^2_{L^2}   .
			\end{split}
		\end{equation}
		
		Let us turn to the term
		$$	I_{22}(h)=\int_{h\delta}^{(h+1)\delta}\int_{h\delta}^{s}\bm{b}(s) \cdot \nabla \big(\bm{b}(r)\cdot \nabla (\xi_r-\xi_{h\delta})\big) \, drds.$$
		For any test function $\phi \in H^{\beta}(\T^2)$ with $\beta>3$, it holds
		$$\big|\big\langle \bm{b}(s) \cdot \nabla \big(\bm{b}(r)\cdot \nabla (\xi_r-\xi_{h\delta})\big), \phi \big \rangle\big|=\big|\big\langle \bm{b}(r) \cdot \nabla(\bm{b}(s) \cdot \nabla \phi) , \xi_r-\xi_{h\delta} \big\rangle \big|.$$
		By Lemma \ref{HCH} and Lemma \ref{interpolation}, for $\gamma \in (0,\frac{1}{2})$, we have
		\begin{equation}\nonumber
			\begin{split}
				&\quad \  \big|\big\langle \bm{b}(r) \cdot \nabla(\bm{b}(s) \cdot \nabla \phi) , \xi_r-\xi_{h\delta} \big\rangle \big|\\
				&\leq \|\bm{b}(r) \cdot \nabla(\bm{b}(s) \cdot \nabla \phi)\|_{H^{\gamma}} \|\xi_r-\xi_{h\delta}\|_{H^{-\gamma}}  \\
				&\lesssim  \|\bm{b}(r)\|_{C^{2\gamma}}  \|\nabla(\bm{b}(s)\cdot \nabla \phi)\|_{H^{\gamma}} \|\xi_r-\xi_{h\delta}\|_{H^{-\gamma}}  \\
				&\lesssim   \|\bm{b}(r)\|_{H^{1+2\gamma}}  \|\nabla(\bm{b}(s)\cdot \nabla \phi)\|_{H^{\gamma}} \|\xi_r-\xi_{h\delta}\|^{\gamma}_{H^{-1}} \|\xi_r-\xi_{h\delta}\|^{1-\gamma}_{L^{2}}  .
			\end{split}
		\end{equation}
		Then for $\gamma \in (0,\frac{1}{3})$, we use \eqref{decompose H^g} to get
		$$\big\|\bm{b}(s) \cdot \nabla \big(\bm{b}(r)\cdot \nabla (\xi_r-\xi_{h\delta})\big)\big\|_{H^{-\beta}} \lesssim \|\bm{b}(r)\|_{H^{2-\gamma}}  \|\bm{b}(s)\|_{H^{2-\gamma}} \|\xi_r-\xi_{h\delta}\|^{\gamma}_{H^{-1}} \|\xi_r-\xi_{h\delta}\|^{1-\gamma}_{L^{2}}. $$
		Hence we have
		\begin{equation}\nonumber
			\begin{split}
				\big\|I_{22}(h)\big\|_{H^{-\beta}}  &\lesssim  \int_{h\delta}^{(h+1)\delta} \int_{h\delta}^{s} \|\bm{b}(r)\|_{H^{2-\gamma}}  \|\bm{b}(s)\|_{H^{2-\gamma}} \|\xi_r-\xi_{h\delta}\|^{\gamma}_{H^{-1}} \|\xi_r-\xi_{h\delta}\|^{1-\gamma}_{L^{2}} \, drds \\
				&\lesssim  \|\xi_0\|^{1-\gamma}_{L^2}  \int_{h\delta}^{(h+1)\delta} \|\bm{b}(r)\|_{H^{2-\gamma}} \|\xi_{r}-\xi_{h\delta}\|^{\gamma}_{H^{-1}}  \, dr \int_{h\delta}^{(h+1)\delta} \|\bm{b}(s)\|_{H^{2-\gamma}} \, ds.
			\end{split}
		\end{equation}
		We take expectation on the above formula and obtain
		\begin{equation}\nonumber
			\begin{split}
				\mathbb{E}\Big[\big\|I_{22}(h)\big\|_{H^{-\beta}} \Big] & \lesssim  \|\xi_0\|^{1-\gamma}_{L^2} \, \bigg[\mathbb{E}\Big( \int_{h\delta}^{(h+1)\delta} \|\bm{b}(r)\|_{H^{2-\gamma}}  \|\xi_{r}-\xi_{h\delta}\|^{\gamma}_{H^{-1}} \, dr \Big)^2\bigg]^{\frac{1}{2}}\\
				&\quad \times \bigg[\mathbb{E}\Big(\int_{h\delta}^{(h+1)\delta} \|\bm{b}(s)\|_{H^{2-\gamma}} \, ds\Big)^2\bigg]^{\frac{1}{2}}.
			\end{split}
		\end{equation}
		Now we use H\"older's inequality to further deal with the two terms respectively. By Lemma \ref{lemma Eb} and Lemma \ref{E H^-1},
		\begin{equation}\label{I22 E1}
			\begin{split}
				&\quad \ \mathbb{E}\bigg( \int_{h\delta}^{(h+1)\delta} \|\bm{b}(r)\|_{H^{2-\gamma}} \|\xi_{r}-\xi_{h\delta}\|^{\gamma}_{H^{-1}}  \, dr \bigg)^2\\
				&\leq \mathbb{E} \bigg[\delta \int_{h\delta}^{(h+1)\delta} \|\bm{b}(r)\|^2_{H^{2-\gamma}} \|\xi_{r}-\xi_{h\delta}\|^{2\gamma}_{H^{-1}} \ dr \bigg]\\
				&\leq \delta \int_{h\delta}^{(h+1)\delta} \Big(\mathbb{E} \big[ \|\bm{b}(r)\|^4_{H^{2-\gamma}}  \big]\Big)^\frac{1}{2}  \Big(\mathbb{E} \big[ \|\xi_{r}-\xi_{h\delta}\|^{4\gamma}_{H^{-1}} \big] \Big)^\frac{1}{2} dr\\
				&\lesssim \delta^{2(1+\gamma)} \nu^{1+\gamma} \alpha^{1+\gamma} \, C^{1/2}_{\theta,2-\gamma,4} \, C^{\gamma}_{\theta,1+\gamma,2} \, \|\xi_0\|^{2\gamma}_{L^2} \big(1+\|\xi_0\|_{L^2} \big)^{2\gamma} ,
			\end{split}
		\end{equation}
		to get the last line, we have used $\mathbb{E} \big[ \|\xi_{r}-\xi_{h\delta}\|^{4\gamma}_{H^{-1}} \big] \leq \Big(\mathbb{E} \big[ \|\xi_{r}-\xi_{h\delta}\|^{2}_{H^{-1}} \big] \Big)^{2\gamma}$. Besides, for the second term, Lemma \ref{lemma Eb} yields		
		\begin{equation}\label{I22 E2}
			\mathbb{E}\bigg(\int_{h\delta}^{(h+1)\delta} \|\bm{b}(s)\|_{H^{2-\gamma}} \, ds\bigg)^2 \leq \, \delta \int_{h\delta}^{(h+1)\delta} \mathbb{E} \big[\|\bm{b}(s)\|^2_{H^{2-\gamma}} \big] ds \,  \lesssim \, \delta^2 \nu \alpha \, C_{\theta,2-\gamma,2}.
		\end{equation}
		Hence we can combine \eqref{I22 E1} and \eqref{I22 E2} to get
		\begin{equation}\nonumber
			\mathbb{E}\Big[\big\|I_{22}(h)\big\|_{H^{-\beta}} \Big] \lesssim \delta^{2+\gamma} \nu^{1+\frac{\gamma}{2}} \alpha^{1+\frac{\gamma}{2}}   \, C^{1/4}_{\theta,2-\gamma,4} \, C^{\gamma/2}_{\theta,1+\gamma,2} \,  C^{1/2}_{\theta,2-\gamma,2} \,   \|\xi_0\|_{L^2} \big(1+\|\xi_0\|_{L^2} \big)^{\gamma} .
		\end{equation}
		Taking supremum, Remark \ref{Jensen C} yields
		\begin{equation}\nonumber
			\begin{split}
				\mathbb{E} \bigg[\sup_{1\leq m<n \leq T/\delta-1}\Big\|\sum_{h=m}^{n}I_{22}(h)\Big\|_{H^{-\beta}}\bigg] &\leq \sum_{h=1}^{T/\delta-1} \mathbb{E}\Big[\big\|I_{22}(h)\big\|_{H^{-\beta}} \Big]\\
				&\lesssim  \delta^{1+\gamma} \nu^{1+\frac{\gamma}{2}} \alpha^{1+\frac{\gamma}{2}}  \, T C^{\frac{1}{2}+\frac{\gamma}{4}}_{\theta,2-\gamma,4}  \, \|\xi_0\|_{L^2} \big(1+\|\xi_0\|_{L^2} \big)^{\gamma} .
			\end{split}
		\end{equation}
		
		As for the term
		$$	I_{24}(h)=-\kappa \int_{h\delta}^{(h+1)\delta}\int_{h\delta}^{s}\bm{b}(s) \cdot \nabla (\Delta \xi_r)\, drds,$$
		notice that for every test function $\phi \in H^{\beta}(\T^2)$, it holds
		$\big|\langle \bm{b}(s) \cdot \nabla (\Delta \xi_r), \phi \rangle\big|=\big|\langle \Delta(\bm{b}(s) \cdot \nabla \phi) , \xi_r \rangle \big| ;$
		and by Lemma \ref{interpolation}, for $\gamma \in (0,1)$, we have
		\begin{equation}\nonumber
			\big|\langle \Delta(\bm{b}(s) \cdot \nabla \phi) , \xi_r \rangle \big| \leq \| \Delta(\bm{b}(s) \cdot \nabla \phi) \|_{H^{\gamma-1}} \| \xi_r\|_{H^{1-\gamma}}  \lesssim \| \nabla(\bm{b}(s) \cdot \nabla \phi) \|_{H^{\gamma}} \| \xi_r \|^{1-\gamma}_{H^1} \| \xi_0 \|^{\gamma}_{L^2}.
		\end{equation}	
		Then we can use \eqref{decompose H^g} to obtain for $\beta>3$ and $\gamma \in (0,\frac{1}{2})$,
		$$\big\|\bm{b}(s) \cdot \nabla (\Delta \xi_r)\big\|_{H^{-\beta}} \lesssim \|\bm{b}(s)\|_{H^{2-\gamma}} \| \xi_r \|^{1-\gamma}_{H^1} \| \xi_0 \|^{\gamma}_{L^2}.$$
		Furthermore, H\"older's inequality and \eqref{priori estimates 1} yield
		\begin{equation}\nonumber
			\begin{split}
				\big\|I_{24}(h)\big\|_{H^{-\beta}} &\lesssim  \kappa \, \| \xi_0 \|^{\gamma}_{L^2} \int_{h\delta}^{(h+1)\delta}  \| \xi_r \|^{1-\gamma}_{H^1}  \,  dr  \int_{h\delta}^{(h+1)\delta}  \| \bm{b}(s)\|_{H^{2-\gamma}}\, ds\\
				&\leq \kappa \,  \| \xi_0 \|^{\gamma}_{L^2} \, \delta^{\frac{1+\gamma}{2}} \Big( \int_{h\delta}^{(h+1)\delta}  \|\xi_r\|^{2}_{H^1} \, dr \Big)^{\frac{1-\gamma}{2}}   \int_{h\delta}^{(h+1)\delta}  \| \bm{b}(s)\|_{H^{2-\gamma}} \,  ds\\
				&\lesssim \kappa^{\frac{1+\gamma}{2}} \delta^{\frac{1+\gamma}{2}} \,  \|\xi_{0}\|_{L^2}  \int_{h\delta}^{(h+1)\delta}  \| \bm{b}(s)\|_{H^{2-\gamma}} \, ds.
			\end{split}
		\end{equation}	
		By Lemma \ref{lemma Eb}, we take expectation and deduce
		\begin{equation}\nonumber
			\begin{split}
				\mathbb{E} \bigg[\sup_{1\leq m<n \leq T/\delta-1}\Big\|\sum_{h=m}^{n}I_{24}(h)\Big\|_{H^{-\beta}}\bigg] &\leq \sum_{h=1}^{T/\delta-1} \mathbb{E}\Big[\big\|I_{24}(h)\big\|_{H^{-\beta}}\Big]\\
				&\lesssim \kappa^{\frac{1+\gamma}{2}} \delta^{\frac{1+\gamma}{2}}  \,  \|\xi_{0}\|_{L^2}  \int_{0}^{T} \Big(\mathbb{E} \big[ \| \bm{b}(s)\|^2_{H^{2-\gamma}} \big]\Big)^{\frac{1}{2}} \, ds \\
				&\lesssim \kappa^{\frac{1+\gamma}{2}} \delta^{\frac{1+\gamma}{2}} \nu^{\frac{1}{2}} \alpha^{\frac{1}{2}} \, T C^{1/2}_{\theta,2-\gamma,2} \,  \|\xi_{0}\|_{L^2} .
			\end{split}
		\end{equation}
		
		Finally, let us estimate the term
		$$	\big\|I_{31}(h)\big\|_{H^{-\beta}} \le \kappa \int_{h\delta}^{(h+1)\delta} \big\| \Delta(\xi_s - \xi_{h\delta}) \big\|_{H^{-\beta}} \, ds \le \kappa \int_{h\delta}^{(h+1)\delta} \| \xi_s - \xi_{h\delta}\|_{H^{-\gamma}} \, ds, $$
		for any $\gamma \in (0,1)$. Then Lemma \ref{interpolation} yields
		\begin{equation}\nonumber
			\begin{split}
				\big\|I_{31}(h)\big\|_{H^{-\beta}} &\leq \kappa \,  \int_{h\delta}^{(h+1)\delta} \|\xi_s-\xi_{h\delta}\|^{\gamma}_{H^{-1}} \|\xi_s-\xi_{h\delta}\|^{1-\gamma}_{L^{2}} \, ds\\
				&\lesssim \kappa \, \|\xi_{0}\|_{L^2}^{1-\gamma} \int_{h\delta}^{(h+1)\delta} \|\xi_s-\xi_{h\delta}\|^{\gamma}_{H^{-1}}\, ds.
			\end{split}
		\end{equation}
		By Lemma \ref{E H^-1}, we arrive at
		\begin{equation}\nonumber
			\begin{split}
				\mathbb{E} \bigg[\sup_{1\leq m<n \leq T/\delta-1}\Big\|\sum_{h=m}^{n}I_{31}(h)\Big\|_{H^{-\beta}}\bigg] &\leq \sum_{h=1}^{T/\delta-1} \mathbb{E}\Big[\big\|I_{31}(h)\big\|_{H^{-\beta}}\Big]\\
				&\lesssim \kappa \,  \|\xi_{0}\|_{L^2}^{1-\gamma} \sum_{h=1}^{T/\delta-1} \int_{h\delta}^{(h+1)\delta} \Big(\mathbb{E} \big[\|\xi_{(h+1)\delta}-\xi_{h\delta}\|^{2}_{H^{-1}}\big]\Big)^{\frac{\gamma}{2}} \, ds\\
				&\lesssim \kappa \delta^{\gamma} \nu^{\frac{\gamma}{2}} \alpha^{\frac{\gamma}{2}} \, T C^{\gamma/2}_{\theta,1+\gamma,2} \,  \|\xi_{0}\|_{L^2} \big(1+\|\xi_{0}\|_{L^2} \big)^{\gamma} .
			\end{split}
		\end{equation}
	\end{proof}

	\subsection{The term $I_{231}(h)$}\label{subsec-231}
	
	In this section, we follow \cite{Pappalettera} and use Nakao's method (cf. \cite{Nakao}) to estimate the term
	\begin{equation}\nonumber
		I_{231}(h)=4\nu \sum_{k,k'\in \mathbb{Z}_0^2} \theta_{k}\sigma_{k} \cdot \nabla( \theta_{k'}\sigma_{k'} \cdot \nabla \xi_{h\delta})  \bigg( \int_{h\delta}^{(h+1)\delta}\int_{h\delta}^{s} \eta^{\alpha,k}(s) \, \eta^{\alpha,{k'}}(r)  \, drds-   \delta_{k,k'} \frac{\delta}{2}\bigg).
	\end{equation}
	
	\begin{lemma}\label{nakao}
		Let $\delta\in(0,1)$ satisfy $\delta^4\alpha^3 \lesssim 1$, then the following inequality holds for any $T\geq 1$:
		$$ \mathbb{E}\bigg[\sup_{1\leq m<n \leq T/\delta-1} \Big\| \sum_{h=m}^{n} I_{231}(h)\Big\|_{H^{-\beta}}\bigg]  \lesssim  \nu \delta^{-1} \alpha^{-1} TD_{\theta,\gamma} \,  \|\xi_{0}\|_{L^{2}} , $$	
		where $D_{\theta,\gamma}:=\big(\sum_{k\in \mathbb{Z}_0^2} |\theta_{k}| \, |k|^{2-\gamma}\big)^2$ is a finite constant depending on $\theta\in \ell^2(\Z^2_0)$ and $\gamma$.
	\end{lemma}
	
	\begin{proof}
		For convenience, we define
		$$	c_{k,k'}(h):=\int_{h\delta}^{(h+1)\delta}\int_{h\delta}^{s} \eta^{\alpha,k}(s) \, \eta^{\alpha,{k'}}(r)  \, drds.$$
		By \cite[Lemma 5.2]{Pappalettera}, the conditional expectation of $c_{k,k'}(h)$ with respect to $\mathcal{F}_{h\delta}$ is
		\begin{equation}\label{conditional expectation of ckk}
			\mathbb{E}\big[c_{k,k'}(h)|\mathcal{F}_{h\delta}\big]=\eta^{\alpha,k}(h\delta) \, \eta^{\alpha,{k'}}(h\delta) \, \frac{(1-e^{-\alpha\delta})^2}{2\alpha^{2}}+ \delta_{k,k'}\bigg[\frac{\delta}{2}+ \frac1\alpha \Big(e^{-\alpha\delta}-1 +\frac{1}{4}(1-e^{-2\alpha\delta}) \Big)\bigg].
		\end{equation}
		
		Now we define two processes as follows:
		$$\aligned	
		M_n&= \sum_{h=1}^{n-1}\sum_{k,k'\in \mathbb{Z}_0^2} \theta_{k}\sigma_{k} \cdot \nabla( \theta_{k'}\sigma_{k'} \cdot \nabla \xi_{h\delta})  \Big( c_{k,k'}(h)- \mathbb{E}\big[c_{k,k'}(h)|\mathcal{F}_{h\delta}\big]\Big), \\
		R_n& = \sum_{h=1}^{n-1}\sum_{k,k'\in \mathbb{Z}_0^2}\theta_{k}\sigma_{k} \cdot \nabla( \theta_{k'}\sigma_{k'} \cdot \nabla \xi_{h\delta})  \Big( \mathbb{E}[c_{k,k'}(h)|\mathcal{F}_{h\delta}]- \delta_{k,k'} \frac{\delta}{2} \Big).
		\endaligned $$
		Notice that $\left\{M_n\right\}_{n=1,\ldots,T/\delta}$ is a $H^{-\beta}$-valued discrete martingale with respect to $\left\{ \mathcal{F}_{n\delta}\right\}_{n=1,\ldots,T/\delta}$, hence by Doob's maximal inequality,
		\begin{equation} \label{supM_n}
			\begin{split}
				&\quad \ \mathbb{E}\bigg[\sup_{1< n \leq T/\delta}\big\|M_n\big\|_{H^{-\beta}}^2\bigg]  \lesssim \mathbb{E}\Big[\big\|M_{T/\delta}\big\|_{H^{-\beta}}^2\Big] \\
				&\leq \sum_{h=1}^{T/\delta-1}\mathbb{E}\, \bigg\|\sum_{k,k'\in \mathbb{Z}_0^2}\theta_{k}\sigma_{k} \cdot \nabla( \theta_{k'}\sigma_{k'} \cdot \nabla \xi_{h\delta})   \Big( c_{k,k'}(h)- \mathbb{E}\big[c_{k,k'}(h)|\mathcal{F}_{h\delta}\big]\Big)\bigg\|_{H^{-\beta}}^2 \\
				&\leq \sum_{h=1}^{T/\delta-1}\mathbb{E} \bigg(\sum_{k,k'\in \mathbb{Z}_0^2} \big\|\theta_{k}\sigma_{k} \cdot \nabla( \theta_{k'}\sigma_{k'} \cdot \nabla \xi_{h\delta}) \big\|_{H^{-\beta}} \Big|c_{k,k'}(h)- \mathbb{E}\big[c_{k,k'}(h)|\mathcal{F}_{h\delta}\big]\Big| \bigg)^2.
			\end{split}
		\end{equation}
		
		We first give the following estimate. For any test function $\phi \in H^{\beta}(\T^2)$, it holds
		$$\big|\big\langle \theta_{k}\sigma_{k} \cdot \nabla( \theta_{k'}\sigma_{k'} \cdot \nabla \xi_{h\delta}) , \phi \big\rangle\big|=\big|\big\langle \theta_{k'}\sigma_{k'} \cdot \nabla(\theta_{k}\sigma_{k} \cdot \nabla \phi),\xi_{h\delta}\big\rangle \big| ;$$
		by Lemma \ref{HHH}, for $\gamma \in (0,\frac{1}{2})$, we use \eqref{decompose H^g} to obtain
		\begin{equation}\nonumber
			\begin{split}
				\big|\big\langle \theta_{k'}\sigma_{k'} \cdot \nabla(\theta_{k}\sigma_{k} \cdot \nabla \phi),\xi_{h\delta}\big\rangle \big| & \lesssim \| \nabla (\theta_k \sigma_{k} \cdot \nabla \phi )\|_{L^{2}} \|\theta_{k'} \sigma_{k'} \xi_{h\delta} \|_{L^{2}} \\
				&\lesssim \|\theta_{k}\sigma_k\|_{H^{2-\gamma}} \|\phi\|_{H^{2+2\gamma}} \|\theta_{k'} \sigma_{k'}\|_{L^{\infty}} \|\xi_{h\delta}\|_{L^{2}}\\
				& \lesssim \|\theta_{k}\sigma_k\|_{H^{2-\gamma}} \|\phi\|_{H^{2+2\gamma}} \|\theta_{k'} \sigma_{k'}\|_{H^{2-\gamma}} \|\xi_{0}\|_{L^{2}}.
			\end{split}
		\end{equation}
		Combining the above two results, we get for $\beta>3$,
		\begin{equation}\label{Mn}
			\big\| \theta_{k}\sigma_{k} \cdot \nabla( \theta_{k'}\sigma_{k'} \cdot \nabla \xi_{h\delta}) \big\|_{H^{-\beta}} \lesssim  \|\theta_{k}\sigma_k\|_{H^{2-\gamma}}  \|\theta_{k'} \sigma_{k'}\|_{H^{2-\gamma}} \|\xi_{0}\|_{L^{2}},
		\end{equation}
		hence,
		\begin{equation}\nonumber
			\begin{split}
				&\quad \  \mathbb{E} \bigg(\sum_{k,k'\in \mathbb{Z}_0^2} \big\|\theta_{k}\sigma_{k} \cdot \nabla( \theta_{k'}\sigma_{k'} \cdot \nabla \xi_{h\delta})\big\|_{H^{-\beta}} \Big|c_{k,k'}(h)- \mathbb{E}\big[c_{k,k'}(h)|\mathcal{F}_{h\delta}\big]\Big| \bigg)^2\\
				&\lesssim  \|\xi_{0}\|^2_{L^{2}} \, \mathbb{E}\bigg(\sum_{k,k'\in \mathbb{Z}_0^2}  \|\theta_{k}\sigma_k\|_{H^{2-\gamma}}  \|\theta_{k'} \sigma_{k'}\|_{H^{2-\gamma}} \, \Big|c_{k,k'}(h)- \mathbb{E}\big[c_{k,k'}(h)|\mathcal{F}_{h\delta}\big]\Big|\bigg)^2.
			\end{split}
		\end{equation}
		We regard each term of $\|\theta_{k}\sigma_k\|_{H^{2-\gamma}}  \|\theta_{k'} \sigma_{k'}\|_{H^{2-\gamma}}$ as the product of their square roots; then the Cauchy-Schwartz inequality and the projective property of conditional expectation yield
		\begin{equation}\nonumber
			\begin{split}
				&\quad \mathbb{E}\bigg(\sum_{k,k'\in \mathbb{Z}_0^2}  \|\theta_{k}\sigma_k\|_{H^{2-\gamma}}  \|\theta_{k'} \sigma_{k'}\|_{H^{2-\gamma}} \, \Big|c_{k,k'}(h)- \mathbb{E}\big[c_{k,k'}(h)|\mathcal{F}_{h\delta}\big]\Big|\bigg)^2 \\
				&\leq \Big(\sum_{k,k'\in \mathbb{Z}_0^2} \|\theta_{k}\sigma_k\|_{H^{2-\gamma}}  \|\theta_{k'} \sigma_{k'}\|_{H^{2-\gamma}} \Big) \\
				&\quad \times \bigg(\sum_{k,k'\in \mathbb{Z}_0^2}\|\theta_{k}\sigma_k\|_{H^{2-\gamma}}  \|\theta_{k'} \sigma_{k'}\|_{H^{2-\gamma}} \mathbb{E}\Big|c_{k,k'}(h)- \mathbb{E}\big[c_{k,k'}(h)|\mathcal{F}_{h\delta}\big]\Big| ^2\bigg)\\
				&\leq \Big(\sum_{k \in \mathbb{Z}_0^2} \|\theta_{k}\sigma_k\|_{H^{2-\gamma}}  \Big)^2 \bigg(\sum_{k,k'\in \mathbb{Z}_0^2}\|\theta_{k}\sigma_k\|_{H^{2-\gamma}}  \|\theta_{k'} \sigma_{k'}\|_{H^{2-\gamma}} \, \mathbb{E}\Big[c_{k,k'}(h)^2\Big] \bigg).
			\end{split}
		\end{equation}
		Notice that the following formula holds:
		\begin{equation}\nonumber
			\begin{split}
				\mathbb{E}\Big[c_{k,k'}(h)^2\Big]&=\mathbb{E}\bigg[\bigg( \int_{h\delta}^{(h+1)\delta}\int_{h\delta}^{s} \eta^{\alpha,k}(s) \, \eta^{\alpha,{k'}}(r) \, drds\bigg)^2 \, \bigg]\\
				&=\mathbb{E}\bigg[\bigg(\int_{h\delta}^{(h+1)\delta}\eta^{\alpha,k}(s) \Big(W_s^{k'}-W_{h\delta}^{k'}- \frac1\alpha \big(\eta^{\alpha,k'}(s)-\eta^{\alpha,k'}(h\delta)\big)\Big)ds\bigg)^2\, \bigg]\\
				&\leq \mathbb{E}\bigg[\delta  \int_{h\delta}^{(h+1)\delta}\big|\eta^{\alpha,k}(s)\big|^2 \Big(W_s^{k'}-W_{h\delta}^{k'}- \frac1\alpha \big(\eta^{\alpha,k'}(s)-\eta^{\alpha,k'}(h\delta)\big)\Big)^2 ds \bigg].
			\end{split}
		\end{equation}
		Again by Cauchy's inequality, we have
		\begin{equation}\nonumber
			\begin{split}
				\mathbb{E}\Big[c_{k,k'}(h)^2\Big]&\leq \delta \int_{h\delta}^{(h+1)\delta} \Big(\mathbb{E}\big[ |\eta^{\alpha,k}(s)|^4\big]\Big)^{\frac{1}{2}} \bigg(\mathbb{E} \Big[ \Big(W_s^{k'}-W_{h\delta}^{k'}- \frac1\alpha \big(\eta^{\alpha,k'}(s)-\eta^{\alpha,k'}(h\delta)\big)\Big)^4\Big]\bigg)^{\frac{1}{2}} ds\\
				&\lesssim \delta \int_{h\delta}^{(h+1)\delta} \alpha \bigg(\mathbb{E} \Big[\big|W_{(h+1)\delta}^{k'}-W_{h\delta}^{k'}\big|^4\Big] + \frac1{\alpha^{4}} \mathbb{E} \Big[\big| \eta^{\alpha,k'}(s)\big|^4 \Big]\bigg)^{\frac{1}{2}} ds\\
				&\lesssim \delta^3 \alpha +\delta^2 .
			\end{split}
		\end{equation}
		Summarizing the above estimates yields
		$$\aligned
		&\quad \ \mathbb{E} \bigg(\sum_{k,k'\in \mathbb{Z}_0^2} \big\|\theta_{k}\sigma_{k} \cdot \nabla( \theta_{k'}\sigma_{k'} \cdot \nabla \xi_{h\delta})\big\|_{H^{-\beta}} \Big|c_{k,k'}(h)- \mathbb{E}\big[c_{k,k'}(h)|\mathcal{F}_{h\delta}\big] \Big| \bigg)^2 \\
		&\lesssim \|\xi_0 \|_{L^2}^2 \Big(\sum_{k \in \mathbb{Z}_0^2} \|\theta_{k}\sigma_k\|_{H^{2-\gamma}} \Big)^2 \bigg(\sum_{k,k'\in \mathbb{Z}_0^2}\|\theta_{k}\sigma_k\|_{H^{2-\gamma}}  \|\theta_{k'} \sigma_{k'}\|_{H^{2-\gamma}} \, (\delta^3 \alpha +\delta^2 ) \bigg) \\
		&= \|\xi_0 \|_{L^2}^2 \Big(\sum_{k \in \mathbb{Z}_0^2} \|\theta_{k}\sigma_k\|_{H^{2-\gamma}} \Big)^4 (\delta^3 \alpha +\delta^2 ) .
		\endaligned $$
		Substituting this estimate into \eqref{supM_n}, we deduce
		\begin{equation}\nonumber
			\begin{split}
				\mathbb{E} \bigg[\sup_{1<n \leq T/\delta}\big\|M_n\big\|_{H^{-\beta}}\bigg]  &\leq \mathbb{E} \bigg[\sup_{1 <n \leq T/\delta}\big\|M_n\big\|_{H^{-\beta}}^2\bigg]^{\frac{1}{2}}\\
				&\lesssim  T^{\frac{1}{2}} \|\xi_{0}\|_{L^{2}} \Big(\sum_{k\in \mathbb{Z}_0^2} \|\theta_{k}\sigma_k\|_{H^{2-\gamma}}\Big)^2 \big(\delta \alpha^{\frac12} + \delta^{\frac12} \big)\\
				&\lesssim \delta \alpha^{\frac{1}{2}} T^{\frac{1}{2}}\|\xi_{0}\|_{L^{2}} \Big(\sum_{k\in \mathbb{Z}_0^2} |\theta_{k}| \, |k|^{2-\gamma}\Big)^2\\
				&= \delta \alpha^{\frac{1}{2}} T^{\frac{1}{2}} D_{\theta,\gamma} \, \|\xi_{0}\|_{L^{2}} \,  .
			\end{split}
		\end{equation}
		
		Now let us turn to the term $R_n$, notice that \eqref{Mn} yields
		\begin{equation}\nonumber
			\begin{split}
				\big\|R_n\big\|_{H^{-\beta}}&\leq \sum_{h=1}^{n-1} \sum_{k,k'\in \mathbb{Z}_0^2} \big\| \theta_{k}\sigma_{k} \cdot \nabla( \theta_{k'}\sigma_{k'} \cdot \nabla \xi_{h\delta})\big\|_{H^{-\beta}} \, \Big| \mathbb{E}\big[c_{k,k'}(h)|\mathcal{F}_{h\delta}\big]- \delta_{k,k'} \frac{\delta}{2} \Big|\\
				&\lesssim   \|\xi_{0}\|_{L^{2}}\sum_{h=1}^{n-1} \sum_{k,k'\in \mathbb{Z}_0^2}  \|\theta_{k}\sigma_k\|_{H^{2-\gamma}}  \|\theta_{k'} \sigma_{k'}\|_{H^{2-\gamma}} \, \Big| \mathbb{E}\big[c_{k,k'}(h)|\mathcal{F}_{h\delta}\big]- \delta_{k,k'} \frac{\delta}{2} \Big|.
			\end{split}
		\end{equation}
		We use the same method as the term $M_n$ to further deal with the above formula as follows:
		\begin{equation}\nonumber
			\begin{split}
				&\quad \ \sum_{k,k'\in \mathbb{Z}_0^2}  \|\theta_{k}\sigma_k\|_{H^{2-\gamma}}  \|\theta_{k'} \sigma_{k'}\|_{H^{2-\gamma}} \, \Big| \mathbb{E}\big[c_{k,k'}(h)|\mathcal{F}_{h\delta}\big]- \delta_{k,k'} \frac{\delta}{2} \Big|\\
				&=\sum_{k,k'\in \mathbb{Z}_0^2}  \|\theta_{k}\sigma_k\|^{1/2}_{H^{2-\gamma}}  \|\theta_{k'} \sigma_{k'}\|^{1/2}_{H^{2-\gamma}}  \|\theta_{k}\sigma_k\|^{1/2}_{H^{2-\gamma}}  \|\theta_{k'} \sigma_{k'}\|^{1/2}_{H^{2-\gamma}} \, \Big| \mathbb{E}\big[c_{k,k'}(h)|\mathcal{F}_{h\delta}\big]- \delta_{k,k'} \frac{\delta}{2} \Big|\\
				&\leq \Big(\sum_{k\in \mathbb{Z}_0^2}  \|\theta_{k}\sigma_k\|_{H^{2-\gamma}} \Big) \bigg(\sum_{k,k'\in \mathbb{Z}_0^2} \|\theta_{k}\sigma_k\|_{H^{2-\gamma}}  \|\theta_{k'} \sigma_{k'}\|_{H^{2-\gamma}} \,  \Big| \mathbb{E}\big[c_{k,k'}(h)|\mathcal{F}_{h\delta}\big]- \delta_{k,k'} \frac{\delta}{2} \Big|^2\, \bigg)^{\frac{1}{2}}.
			\end{split}
		\end{equation}
		By \eqref{conditional expectation of ckk}, we can easily get
		\begin{equation}\nonumber
			\begin{split}
				&\quad \ \mathbb{E} \bigg[ \Big| \mathbb{E}\big[c_{k,k'}(h)|\mathcal{F}_{h\delta}\big]-\delta_{k,k'} \frac{\delta}{2} \Big|^2  \bigg] \\
				&= \mathbb{E} \bigg[ \Big| \eta^{\alpha,k}(h\delta) \eta^{\alpha,{k'}}(h\delta) \frac{(1-e^{-\alpha\delta})^2}{2\alpha^{2}}+ \delta_{k,k'} \frac1\alpha \big(e^{-\alpha\delta}-1 +\frac{1}{4}(1-e^{-2\alpha\delta})\big) \Big|^2  \bigg]\\
				&\lesssim \alpha^{-2}.
			\end{split}
		\end{equation}
		Combining the above results, we take supremum and then expectation on $\big\|R_n\big\|_{H^{-\beta}}$ to get
		\begin{equation}\nonumber
			\begin{split}
				\mathbb{E} \bigg[\sup_{1 <n \leq T/\delta}\big\|R_n\big\|_{H^{-\beta}}\bigg]
				& \lesssim \alpha^{-1} \|\xi_{0}\|_{L^{2}} \sum_{h=1}^{T/\delta-1} \Big(\sum_{k\in \mathbb{Z}_0^2}  \|\theta_{k}\sigma_k\|_{H^{2-\gamma}} \Big)^2 \\
				&\lesssim  \delta^{-1} \alpha^{-1} \, T D_{\theta,\gamma} \, \|\xi_{0}\|_{L^{2}}.
			\end{split}
		\end{equation}
		
		Taking our assumptions on the parameters into consideration, and noticing that
		$$\sum_{h=m}^{n}I_{231}(h)=4\nu \big( M_{n+1}+R_{n+1}-M_m-R_m \big) ,$$
		we complete the proof of Lemma \ref{nakao}.
	\end{proof}

	\subsection{The term $I_{25}(h)$}\label{subsec-25}

	In this section, we focus on the term
	$$	I_{25}(h)=-\int_{h\delta}^{(h+1)\delta}\bm{b}(s) \cdot \nabla \xi_{h\delta} \, ds =-\Big( \int_{h\delta}^{(h+1)\delta}\bm{b}(s)  \, ds\Big) \cdot \nabla \xi_{h\delta} .$$
	
	According to the definition of $\bm{b}(s)$, it holds
	$$\int_{h\delta}^{(h+1)\delta}\bm{b}(s)\, ds=2\sqrt{\nu}\sum_{k\in \mathbb{Z}_0^2}\theta_k\sigma_k\big(W_{(h+1)\delta}^k-W_{h\delta}^k\big)-2\sqrt{\nu}\alpha^{-1} \sum_{k\in \mathbb{Z}_0^2}\theta_k\sigma_k \big(\eta^{\alpha,k}((h+1)\delta)-\eta^{\alpha,k}(h\delta)\big),$$
	then we can further decompose $I_{25}(h)$ as follows:
	\begin{equation}\nonumber
		\begin{split}
			I_{25}(h)&=-2\sqrt{\nu} \sum_{k\in \mathbb{Z}_0^2} \int_{h\delta}^{(h+1)\delta} \theta_{k} \sigma_{k} \cdot \nabla \xi_{h\delta}  \, dW^k_s\\
			&\quad+2\sqrt{\nu} \alpha^{-1}\sum_{k\in \mathbb{Z}_0^2}\big(\theta_k\sigma_k\cdot \nabla \xi_{h\delta}\big) \big(\eta^{\alpha,k}((h+1)\delta)-\eta^{\alpha,k}(h\delta)\big)\\
			&=:I_{251}(h)+I_{252}(h).
		\end{split}
	\end{equation}
	
	The following lemma gives the result for the term $I_{252}(h)$, as for the term $I_{251}(h)$, we will separately discuss it after the proof of Lemma \ref{I25}.
	
	\begin{lemma}\label{I25}
		Let $\gamma \in (0,\frac{1}{2})$, $\beta>3$ and $T\geq 1$, then we have
		\begin{equation}\nonumber
			\begin{split}
				&\quad	\mathbb{E}\bigg[\sup_{1\leq m<n \leq T/\delta-1}\Big\| \sum_{h=m}^{n}I_{252}(h) \Big\|_{H^{-\beta}}\bigg]\\
				&\lesssim \nu^{\frac{2+3\gamma}{2(1+\gamma)}} \big( \delta^{-\frac{\gamma}{2(1+\gamma)}}  \alpha^{-\frac{\gamma}{2(1+\gamma)}} +  \delta^{-\frac{\gamma}{1+\gamma}} \alpha^{-\frac{\gamma}{1+\gamma}}  \big)\, T C^{\frac{2+3\gamma}{4(1+\gamma)}}_{\theta,2-\gamma,4} \,  \|\xi_0\|_{L^{2}} \big(1+\|\xi_0\|_{L^{2}}\big)^2 \\
				&\quad+\nu^{\frac{1}{2}} \alpha^{-\frac{1}{2}}  \, C^{1/2}_{\theta,2-\gamma,2} \log^{\frac{1}{2}}(1+\alpha T)  \, \|\xi_0\|_{L^{2}}.
			\end{split}
		\end{equation}
	\end{lemma}

	\begin{proof}
		We first reformulate the sum as follows:
		\begin{equation}\nonumber
			\begin{split}
				\sum_{h=m}^{n}I_{252}(h)&= \alpha^{-1} \sum_{h=m}^{n} \big(\bm{b}((h+1)\delta) -\bm{b}(h \delta) \big)\cdot \nabla\xi_{h\delta} \\
				&=-\alpha^{-1} \bigg[\sum_{h=m+1}^{n}  \bm{b}(h\delta) \cdot \nabla (\xi_{h\delta}-\xi_{(h-1)\delta} )+  \bm{b}(m\delta)\cdot\nabla \xi_{m\delta} - \bm{b} ((n+1)\delta )\cdot \nabla \xi_{n\delta} \bigg].
			\end{split}
		\end{equation}
		We will estimate each term of the above formula respectively. For the first term, notice that for any test function $\phi \in H^{\beta}(\T^2)$, it holds
		$$\big| \big\langle\bm{b}(h\delta) \cdot \nabla(\xi_{h\delta}-\xi_{(h-1)\delta}),\phi \big\rangle \big|= \big| \big\langle \bm{b}(h\delta) \cdot \nabla \phi ,\xi_{h\delta}-\xi_{(h-1)\delta} \big\rangle \big| ;$$
		meanwhile,	\eqref{decompose H^g} yields
		\begin{equation*}
			\begin{split}
				\big| \big\langle \bm{b}(h\delta) \cdot \nabla \phi ,\xi_{h\delta}-\xi_{(h-1)\delta} \big\rangle \big|
				&\lesssim  \|\nabla(\bm{b}(h\delta) \cdot \nabla \phi )\|_{H^{\gamma}} \|\xi_{h\delta}-\xi_{(h-1)\delta}\|_{H^{-1-\gamma}}\\
				&\lesssim \|\bm{b}(h\delta)\|_{H^{2-\gamma}} \|\phi\|_{H^{2+2\gamma}} \|\xi_{h\delta}-\xi_{(h-1)\delta}\|_{H^{-1-\gamma}}.
			\end{split}
		\end{equation*}
		Hence we can further get for $\beta>3$ and $\gamma \in (0,\frac{1}{2})$,
		$$\big\|\bm{b}(h\delta) \cdot \nabla (\xi_{h\delta}-\xi_{(h-1)\delta} ) \big\|_{H^{-\beta}} \lesssim \|\bm{b}(h\delta)\|_{H^{2-\gamma}}  \|\xi_{h\delta}-\xi_{(h-1)\delta}\|_{H^{-1-\gamma}}. $$
		As for the second term, we can use Sobolev embedding theorem to get for  $\phi \in H^{\beta}(\T^2)$,
		\begin{equation*}
			\begin{split}
				\big|\langle \bm{b}(m\delta) \cdot \nabla  \xi_{m\delta}, \phi \rangle\big|=\big|\langle \bm{b}(m\delta) \cdot \nabla \phi, \xi_{m\delta} \rangle\big| &\lesssim \|\bm{b}(m\delta)\|_{L^{2}} \|\nabla\phi\|_{L^{\infty}}  \|\xi_{m\delta}\|_{L^{2}} \\
				&\lesssim  \|\bm{b}(m\delta)\|_{H^{2-\gamma}} \|\phi\|_{H^{2+\gamma}}  \|\xi_{0}\|_{L^{2}} .
			\end{split}
		\end{equation*}
		Hence $\big\|\bm{b}(m\delta) \cdot \nabla  \xi_{m\delta} \big\|_{H^{-\beta}} \lesssim \|\bm{b}(m\delta)\|_{H^{2-\gamma}}  \|\xi_{0}\|_{L^{2}}$ for any $\beta>3$ and $\gamma \in (0,1)$.
		
		Besides, the third term can be estimated in the same way as the second one and therefore we can get the similar result. Summarizing the above estimates, we obtain
		\begin{equation}\nonumber
			\begin{split}
				\Big\| \sum_{h=m}^{n}I_{252}(h) \Big\|_{H^{-\beta}} &\lesssim \alpha^{-1} \bigg( \sum_{h=m+1}^{n} \|\bm{b}(h\delta)\|_{H^{2-\gamma}}  \|\xi_{h\delta}-\xi_{(h-1)\delta}\|_{H^{-1-\gamma}} \\
				&\quad+ \|\bm{b}(m\delta)\|_{H^{2-\gamma}} \|\xi_0\|_{L^{2}} + \|\bm{b}((n+1)\delta)\|_{H^{2-\gamma}} \|\xi_0\|_{L^{2}}   \bigg).
			\end{split}
		\end{equation}
		Then we take supremum and then expectation to further get
		\begin{equation}\label{sup I_252}
			\begin{split}
				&\quad \ \mathbb{E}\bigg[\sup_{1\leq m<n \leq T/\delta-1}\Big\| \sum_{h=m}^{n}I_{252}(h) \Big\|_{H^{-\beta}}\bigg]\\
				&\lesssim \alpha^{-1}  \sum_{h=1}^{T/\delta-1} \mathbb{E} \Big[ \|\bm{b}(h\delta)\|_{H^{2-\gamma}}  \|\xi_{h\delta}-\xi_{(h-1)\delta}\|_{H^{-1-\gamma}}\Big]+ \alpha^{-1}  \|\xi_0\|_{L^{2}} \,  \mathbb{E}\Big[\sup_{1 \leq m<T/\delta-1} \|\bm{b}(m\delta)\|_{H^{2-\gamma}} \Big].
			\end{split}
		\end{equation}
		By the Cauchy-Schwarz inequality, we obtain
		\begin{equation}\nonumber
			\mathbb{E} \Big[ \|\bm{b}(h\delta)\|_{H^{2-\gamma}}  \|\xi_{h\delta}-\xi_{(h-1)\delta}\|_{H^{-1-\gamma}}\Big]\leq \Big(\mathbb{E} \big[\|\bm{b}(h\delta)\|^2_{H^{2-\gamma}} \big]\Big)^{\frac{1}{2}} \Big(\mathbb{E}\big[ \|\xi_{h\delta}-\xi_{(h-1)\delta}\|^2_{H^{-1-\gamma}}\big]\Big)^{\frac{1}{2}}.
		\end{equation}
		Considering the latter expectation, for $\gamma \in (0,1)$, Lemma \ref{interpolation} yields
		\begin{equation}\nonumber
			\begin{split}
				\mathbb{E}\Big[ \|\xi_{h\delta}-\xi_{(h-1)\delta}\|^2_{H^{-1-\gamma}} \Big]
				&\leq \mathbb{E}\Big[\|\xi_{h\delta}-\xi_{(h-1)\delta}\|^{\frac{2}{1+\gamma}}_{H^{-1}} \, \|\xi_{h\delta}-\xi_{(h-1)\delta}\|^{\frac{2\gamma}{1+\gamma}}_{H^{-2-\gamma}} \Big] \\
				&\leq  \bigg(\mathbb{E}\Big[\|\xi_{h\delta}-\xi_{(h-1)\delta}\|^{\frac{4}{1+\gamma}}_{H^{-1}}  \Big] \bigg)^{\frac{1}{2}} \bigg(\mathbb{E}\Big[\|\xi_{h\delta}-\xi_{(h-1)\delta}\|^{\frac{4\gamma}{1+\gamma}}_{H^{-2-\gamma}}  \Big] \bigg)^{\frac{1}{2}};
			\end{split}
		\end{equation}
		moreover, to apply Lemma \ref{E H2-g}, we need to further estimate the second expectation of the last line as follows:
		$$\mathbb{E}\Big[\|\xi_{h\delta}-\xi_{(h-1)\delta}\|^{\frac{4\gamma}{1+\gamma}}_{H^{-2-\gamma}}  \Big] \leq \mathbb{E}\Big[\|\xi_{h\delta}-\xi_{(h-1)\delta}\|^{2}_{H^{-2-\gamma}}  \Big]^{\frac{2\gamma}{1+\gamma}}.$$
		Combining the above results together, Lemma \ref{E H^-1}, Lemma \ref{E H2-g}  and Remark \ref{Jensen C} yield
		\begin{equation}\label{I252 Eb sum}
			\begin{split}
				&\quad \ \mathbb{E} \Big[\|\bm{b}(h\delta)\|_{H^{2-\gamma}}  \|\xi_{h\delta}-\xi_{(h-1)\delta}\|_{H^{-1-\gamma}}\Big] \\
				&\lesssim \nu^{\frac{2+3\gamma}{2(1+\gamma)}} \Big( \delta^{\frac{2+\gamma}{2(1+\gamma)}}  \alpha^{\frac{2+\gamma}{2(1+\gamma)}} +\delta^{\frac{1}{1+\gamma}} \alpha^{\frac{1}{1+\gamma}}  \Big)\, C^{\frac{2+3\gamma}{4(1+\gamma)}}_{\theta,2-\gamma,4} \, \|\xi_0\|_{L^{2}} \big(1+\|\xi_0\|_{L^2}\big)^2.
			\end{split}
		\end{equation}
		In addition, according to Lemma \ref{Eb sup lemma}, we have
		\begin{equation}\label{I252 sup Eb}
			\mathbb{E}\Big[\sup_{1 \leq m<T/\delta-1} \|\bm{b}(m\delta)\|_{H^{2-\gamma}} \Big] \leq \mathbb{E}\Big[\sup_{1 \leq m<T/\delta-1} \|\bm{b}(m\delta)\|^2_{H^{2-\gamma}} \Big]^{\frac{1}{2}} \lesssim \nu^{\frac{1}{2}} \alpha^{\frac{1}{2}} \, C^{1/2}_{\theta,2-\gamma,2}  \log^{\frac{1}{2}}(1+\alpha T)  .
		\end{equation}
		Inserting \eqref{I252 Eb sum} and \eqref{I252 sup Eb} into \eqref{sup I_252}, we complete the proof.
	\end{proof}
	
	Now let us deal with the term
	\begin{equation}\nonumber
		\sum_{h=m}^{n-1} I_{251}(h)=-2\sqrt{\nu} \sum_{k \in \mathbb{Z}_0^2} \int_{m\delta}^{n\delta}  \theta_{k} \sigma_{k} \cdot \nabla \xi_{[s]} \, dW^k_s,
	\end{equation}
	where $[s]$ is defined as at the beginning of Section 5.1. By the Burkholder-Davis-Gundy's inequality, we have
	\begin{equation}\label{I252 BDG}
		\mathbb{E}\bigg[\Big\|2\sqrt{\nu}\sum_{k \in \mathbb{Z}_0^2}\int_{t_1}^{t_2}  \theta_{k} \sigma_{k} \cdot \nabla \xi_{[s]}   \, dW^k_s\Big\|_{H^{-\beta}}^4\bigg] \lesssim \nu^2 \, \mathbb{E} \bigg[\Big(\sum_{k \in \mathbb{Z}_0^2} \int_{t_1}^{t_2} \big\| \theta_{k}\sigma_{k} \cdot \nabla \xi_{[s]} \big\|_{H^{-\beta}}^2 \, ds \Big)^2\bigg].
	\end{equation}
	Let $e_k(x)=e^{2\pi i k \cdot x}$ and recall the definition of $\sigma_k$, then we have
	$$\sum_{k \in \mathbb{Z}_0^2} \big\| \theta_{k}\sigma_{k} \cdot \nabla \xi_{[s]} \big\|_{H^{-\beta}}^2    \leq \|\theta\|^2_{\ell^{\infty}} \sum_{k \in \mathbb{Z}_0^2} \big\|\sigma_k \, \xi_{[s]}\big\|_{H^{-\beta+1}}^2 \leq \|\theta\|^2_{\ell^{\infty}} \sum_{k \in \mathbb{Z}_0^2} \big\|e_k \, \xi_{[s]}\big\|_{H^{-\beta+1}}^2 ;$$
	furthermore, for $\beta>3$, it holds
	$$\sum_{k \in \mathbb{Z}_0^2} \big\|e_k \, \xi_{[s]}\big\|_{H^{-\beta+1}}^2 \lesssim \sum_{k \in \mathbb{Z}_0^2} \sum_{l \in \mathbb{Z}_0^2} \frac{1}{|l|^{2(\beta-1)}} \big|\big\langle \xi_{[s]}, e_{l-k} \big\rangle\big|^2=\big\|\xi_{[s]}\big\|^2_{L^2} \sum_{l \in \mathbb{Z}_0^2} \frac{1}{|l|^{2(\beta-1)}}.$$
	Combing the above two estimates, we obtain
	$$\sum_{k \in \mathbb{Z}_0^2} \big\| \theta_{k}\sigma_{k} \cdot \nabla \xi_{[s]} \big\|_{H^{-\beta}}^2  \lesssim \|\theta\|^2_{\ell^{\infty}} \|\xi_0\|^2_{L^2}  \sum_{l \in \mathbb{Z}_0^2} \frac{1}{|l|^{2(\beta-1)}} \lesssim_{\beta} \|\theta\|^2_{\ell^{\infty}} \|\xi_0\|^2_{L^2},$$
	where we have used $\sum_{l \in \mathbb{Z}_0^2} |l|^{-2(\beta-1)} < \infty$. Hence \eqref{I252 BDG} yields
	$$\mathbb{E}\bigg[\Big\|2\sqrt{\nu}\sum_{k \in \mathbb{Z}_0^2}\int_{t_1}^{t_2}  \theta_{k} \sigma_{k} \cdot \nabla \xi_{[s]}   \, dW^k_s\Big\|_{H^{-\beta}}^4\bigg] \lesssim \nu^2 \|\theta\|^4_{\ell^{\infty}} \|\xi_0\|^4_{L^2} \, |t_2-t_1|^2.$$
	By the Kolmogorov Continuity Theorem, for every $\rho\in(0,\frac{1}{4})$, we arrive at
	\begin{equation}\label{I252 final}
		\mathbb{E}\Bigg[\sup_{0<t_1<t_2<T} \frac{\Big\| 2\sqrt{\nu}\sum_{k\in \mathbb{Z}_0^2} \int_{t_1}^{t_2} \, \theta_{k} \sigma_{k} \cdot \nabla \xi_{[s]} \, dW^k_s\Big\|_{H^{-\beta}}}{{|t_2-t_1|}^{\rho}}\Bigg] \lesssim \nu^{\frac{1}{2}} T^{\frac{1}{2}-\rho} \|\theta\|_{\ell^{\infty}} \|\xi_0\|_{L^2}  .
	\end{equation}
	
	\subsection{The term $I_1(h)+\delta  (u_{h\delta} \cdot \nabla  \xi_{h\delta})  $} \label{subsec-1-drift}

	For the remaining two terms of \eqref{decompose 2}, we will treat them together and prove
	\begin{lemma}\label{remaing term}
		For $\gamma \in (0,1)$ and $\beta>3$, the following estimate holds for all $T\geq 1$:
		$$\mathbb{E}\bigg[\sup_{1\leq m<n \leq T/\delta-1}	\Big\| \sum_{h=m}^{n}\big(I_1(h)+\delta  (u_{h\delta} \cdot \nabla  \xi_{h\delta})  \big) \Big\|_{H^{-\beta}}\bigg] \lesssim \delta \nu^{\frac{1}{2}}  \alpha^{\frac{1}{2}} \, T C^{1/2}_{\theta,1+\gamma,2} \,   \|\xi_{0}\|^2_{L^{2}} \big(1+\|\xi_{0}\|_{L^{2}}\big)  .$$
	\end{lemma}
	
	\begin{proof}
		For the convenience of calculation, we make the following decomposition:
		\begin{equation}\nonumber
			\begin{split}
				I_1(h)+\delta  (u_{h\delta} \cdot \nabla  \xi_{h\delta}) &=\int_{h\delta}^{(h+1)\delta} \big( u_{h\delta} \cdot \nabla \xi_{h\delta} -u_s \cdot \nabla \xi_s \big) \, ds\\
				&=\int_{h\delta}^{(h+1)\delta} (u_{h\delta}-u_s) \cdot \nabla \xi_{h\delta}  \, ds+\int_{h\delta}^{(h+1)\delta}  u_s \cdot \nabla(\xi_{h\delta} -\xi_s) \, ds  \\
				&=:I_{11}(h)+I_{12}(h).
			\end{split}
		\end{equation}
		
		We first consider the term $I_{11}(h)$. Notice that for every test function $\phi \in H^{\beta}(\T^2)$, it holds
		$\big|\big\langle (u_{h\delta}-u_s) \cdot \nabla \xi_{h\delta}, \phi \big\rangle\big|=\big|\big \langle u_{h\delta}-u_s, \xi_{h\delta} \, \nabla \phi \big \rangle \big| $; besides, by Sobolev embedding theorem, for $\gamma\in (0,1)$, we have
		$$\big|\big \langle u_{h\delta}-u_s, \xi_{h\delta} \, \nabla \phi \big \rangle \big|  \leq \|u_{h\delta}-u_s\|_{L^2} \|\xi_{h\delta}\|_{L^2} \|\nabla \phi\|_{L^{\infty}} \lesssim \|\xi_{h\delta}-\xi_s\|_{H^{-1}} \|\xi_0\|_{L^2} \| \phi\|_{H^{2+\gamma}}.$$
		Then for $\beta>3$, we can further get
		\begin{equation}\nonumber
			\big\|I_{11}(h)\big\|_{H^{-\beta}} \leq \int_{h\delta}^{(h+1)\delta} \big\| (u_{h\delta}-u_s) \cdot \nabla \xi_{h\delta} \big\|_{H^{-\beta}} \, ds\lesssim  \|\xi_{0}\|_{L^{2}}\int_{h\delta}^{(h+1)\delta} \|\xi_{h\delta}-\xi_s\|_{H^{-1}} \, ds.
		\end{equation}
		Taking supremum and then expectation, Lemma \ref{E H^-1} yields
		\begin{equation}\label{I11}
			\begin{split}
				\mathbb{E} \bigg[\sup_{1\leq m<n \leq T/\delta-1}\Big\|\sum_{h=m}^{n}I_{11}(h)\Big\|_{H^{-\beta}}\bigg] &\leq \sum_{h=1}^{T/\delta-1} \mathbb{E}\Big[\big\|I_{11}(h)\big\|_{H^{-\beta}}\Big]\\
				&\lesssim  \|\xi_{0}\|_{L^2} \sum_{h=1}^{T/\delta-1} \int_{h\delta}^{(h+1)\delta}\Big( \mathbb{E} \big[\|\xi_{h\delta}-\xi_s\|_{H^{-1}}^2 \big]\Big)^{\frac{1}{2}} \, ds\\
				&\lesssim \delta \nu^{\frac{1}{2}}  \alpha^{\frac{1}{2}} \, T C^{1/2}_{\theta,1+\gamma,2} \,   \|\xi_{0}\|^2_{L^{2}} \big(1+\|\xi_{0}\|_{L^{2}}\big)  .
			\end{split}
		\end{equation}
		
		As for the term $I_{12}(h)$, we can use  \eqref{L2 decompose} to estimate it as follows: for any $\phi\in H^{2+\gamma}$,
		\begin{equation}\nonumber
			\begin{split}
				\big|\big\langle u_s \cdot \nabla(\xi_{h\delta} -\xi_s), \phi \big\rangle\big|=\big|\big \langle u_s \cdot \nabla \phi, \xi_{h\delta}-\xi_s \big \rangle \big| &\lesssim \|\nabla(u_s \cdot \nabla \phi) \|_{L^{2}} \, \|\xi_{h\delta}-\xi_s\|_{H^{-1}} \\
				&\lesssim \|\xi_0\|_{L^2} \|\phi\|_{H^{2+\gamma}} \|\xi_{h\delta}-\xi_s\|_{H^{-1}} .
			\end{split}
		\end{equation}
		Hence for $\beta>3$, we have
		\begin{equation}\nonumber
			\big\|I_{12}(h)\big\|_{H^{-\beta}} \leq \int_{h\delta}^{(h+1)\delta} \big\| u_s \cdot \nabla(\xi_{h\delta} -\xi_s)\big\|_{H^{-\beta}} \, ds\lesssim  \|\xi_{0}\|_{L^{2}}\int_{h\delta}^{(h+1)\delta} \|\xi_{h\delta}-\xi_s\|_{H^{-1}} \, ds.
		\end{equation}
		Thus we can get the same estimate as the term $I_{11}(h)$,  that is
		\begin{equation}\label{I12}
			\mathbb{E}\bigg[\sup_{1\leq m<n \leq T/\delta-1}	\Big\| \sum_{h=m}^{n}I_{12}(h) \Big\|_{H^{-\beta}}\bigg] \lesssim \delta \nu^{\frac{1}{2}}  \alpha^{\frac{1}{2}} \, T C^{1/2}_{\theta,1+\gamma,2} \,   \|\xi_{0}\|^2_{L^{2}} \big(1+\|\xi_{0}\|_{L^{2}}\big)  .
		\end{equation}
		Lemma \ref{remaing term} follows by combining \eqref{I11} and \eqref{I12}.
	\end{proof}

	\subsection{The terms $I_a$ and $I_b$} \label{subsec-a-b}

Recall the definitions of $I_a$ and $I_b$ at the beginning of Section \ref{subsec-decomp}; since their treatments are similar to those involving $I_{31}(h)$ and $I_1(h)+\delta(u_{h\delta}\cdot \nabla \xi_{h\delta})$, respectively, we omit them here to save space.

	\begin{lemma}\label{Ia and Ib}
		Let $T\geq 1$, $\beta>3$ and $\gamma\in (0,1)$, then the following estimates hold:
		$$\aligned
		\E \bigg[\sup_{1 \leq m<n \leq T/\delta-1} \|I_a\|_{H^{-\beta}}\bigg] &\lesssim \delta^{\gamma} \alpha^{\frac{\gamma}{2}} (\kappa\nu^{\frac{\gamma}{2}} +\nu^{1+\frac{\gamma}{2}} ) \, T C_{\theta,1+\gamma,2}^{\gamma/2} \|\xi_0\|_{L^2} \big(1+\|\xi_0\|_{L^2}\big)^{\gamma},\\
		\E \bigg[\sup_{1 \leq m<n \leq T/\delta-1} \|I_b\|_{H^{-\beta}}\bigg] &\lesssim \delta \nu^{\frac{1}{2}}  \alpha^{\frac{1}{2}} \, T C^{1/2}_{\theta,1+\gamma,2} \,   \|\xi_{0}\|^2_{L^{2}} \big(1+\|\xi_{0}\|_{L^{2}}\big).
		\endaligned $$
	\end{lemma}
	
	\subsection{Proof of Proposition \ref{main proposition}}\label{subs-proof-Prop-3.3}

	Now we will combine the results of Lemmas \ref{prop3 1st lemma}-\ref{Ia and Ib} and prove
	\begin{equation} \nonumber
		\begin{split}
			&\quad \mathbb{E}\bigg[\sup_{1\leq m<n \leq T/\delta-1} \frac1{(|n-m|\delta)^\rho } \Big\|\xi_{n\delta}-\xi_{m\delta}-(\kappa+\nu)\int_{m\delta}^{n\delta} \Delta \xi_s \, ds+\int_{m\delta}^{n\delta} u_s\cdot \nabla \xi_s \, ds \Big\|_{H^{-\beta}} \bigg]\\
			&\lesssim T \|\xi_0\|_{L^2} \big(1+\|\xi_0\|_{L^{2}}\big)^2 \big(\nu^{1+\frac{\gamma}{2}}\alpha^{-\epsilon}  +\nu^{\frac{1}{2}}\|\theta\|_{\ell^{\infty}}\big).
		\end{split}
	\end{equation}
	Recalling the decompositions in \eqref{decomposition of fn-fm} and \eqref{decompose 2}, let us start with terms other than $I_{251}(h)$. Denote
	$$I(h):= I_{1}(h)+I_{21}(h)+I_{22}(h)+I_{231}(h)+I_{24}(h)+I_{252}(h)+I_{31}(h)+\delta ( u_{h\delta} \cdot \nabla  \xi_{h\delta} ),$$
	then we want to prove
	\begin{equation} \label{prf of prop3}
		\mathbb{E}\bigg[\sup_{1\leq m<n \leq T/\delta-1}  \frac1{(|n-m|\delta)^\rho }\Big\|\sum_{h=m}^{n}  I(h)+I_a+I_b \Big\|_{H^{-\beta}} \bigg] \lesssim \nu^{1+\frac{\gamma}{2}}\alpha^{-\epsilon} \, T  \|\xi_0\|_{L^2}\big(1+\|\xi_0\|_{L^{2}}\big)^2.
	\end{equation}
	
	As $\kappa<1$, then we can magnify it to 1 for convenience. Besides, observe that $\nu>1$ and the exponents of $\nu$ are all smaller than $1+\frac{\gamma}{2}$ in estimates in Sections 5.2--5.6, hence we keep $\nu^{1+\frac{\gamma}{2}}$ in the final result. Moreover, once $\theta\in \ell^2(\Z^d_0)$ is fixed, $C_{\theta,\tau,p}$ and $D_{\theta,\gamma}$ are finite constants which do not play a big role in our main results. As for the parts involving the $L^2$-norm of initial value $\xi_0$, they are all dominated by $\|\xi_0\|_{L^2}\big(1+\|\xi_0\|_{L^{2}}\big)^2$, so we mainly focus on parameters $\delta$ and $\alpha$ in the proof.
	
	To make the term $I_{22}(h)$ satisfy \eqref{prf of prop3}, for fixed $\theta\in \ell^2(\Z^2_0)$, we need to make a restriction on $\delta$ and $\alpha$ as follows:
	\begin{equation}\label{1st restriction}
		\delta^{1+\gamma-\rho} \alpha^{1+\frac{\gamma}{2}+\epsilon} \lesssim 1.
	\end{equation}
	Moreover, to apply Lemmas \ref{E H^-1} and \ref{E H2-g} in the proofs of Lemmas \ref{prop3 1st lemma}--\ref{Ia and Ib}, the following conditions are necessary:
	\begin{equation}\label{2st restriction}
		\delta^4 \alpha^3 \lesssim 1 , \quad \delta \alpha \gtrsim1.
	\end{equation}
	
	In order to verify that the above two conditions are consistent with each other, we present the specific choice of parameters. Indeed, for fixed $\gamma \in (0,\frac{1}{3})$, there exist sufficiently small $\rho,\, \epsilon>0$ such that
	$$\epsilon +\rho < \frac14 (1 -\rho -2\epsilon ) \gamma \quad \Leftrightarrow\quad (1+\rho)\Big(1+ \frac\gamma2 +\epsilon \Big) < (1-\epsilon)(1+\gamma-\rho); $$
	therefore, for $\alpha$ big enough, we can choose $\delta$ satisfying
	\begin{equation}\label{delta restriction}
		\alpha^{-\frac{1-\epsilon}{1+\rho}} \lesssim \delta \lesssim \alpha^{-\frac{1+\gamma/2+\epsilon}{1+\gamma-\rho}}.
	\end{equation}
	With this choice, it is easy to see that \eqref{1st restriction} is verified. Next, we deduce from $\alpha^{-1} \le  \alpha^{-\frac{1-\epsilon}{1+\rho}} \lesssim \delta$ that the second condition of \eqref{2st restriction} also holds. Finally, one has $\delta^4 \alpha^3 \lesssim \alpha^{\frac{-1+\gamma-4\epsilon-3\rho}{1+\gamma-\rho}} \leq 1$ as the exponent is negative, which yields the first inequality of \eqref{2st restriction}.
	
	Based on the previous discussions, we choose several terms in the decomposition of $I(h)$ as examples to show that \eqref{prf of prop3} holds under the condition \eqref{delta restriction}, which implies \eqref{1st restriction} and \eqref{2st restriction}.
	
	(i) For the term $I_{24}(h)$, we only need to prove $\delta^{\frac{1+\gamma}{2}-\rho} \alpha^{\frac{1}{2}+\epsilon}\lesssim 1.$ By \eqref{1st restriction} and \eqref{2st restriction}, we immediately deduce
	$$\delta^{\frac{1+\gamma}{2}-\rho} \alpha^{\frac{1}{2}+\epsilon}=\big(\delta^{1+\gamma-\rho} \alpha^{1+\frac{\gamma}{2}+\epsilon} \big) \big(\delta\alpha\big)^{-\frac{1+\gamma}{2}} \lesssim 1;$$
	
	(ii) For the term $I_{231}(h)$, we want to verify that $\delta^{-1-\rho} \alpha^{-1+\epsilon} \lesssim 1$. Recalling our choice of $\delta$ in \eqref{delta restriction}, we can further get
	$$ \delta^{-1-\rho} \alpha^{-1+\epsilon} \lesssim \alpha^{1-\epsilon} \alpha^{-1+\epsilon} =1. $$
	
	(iii) For the term $I_{252}(h)$, we only discuss the latter part here, that is:
	$$\delta^{-\rho} \alpha^{-\frac{1}{2}+\epsilon} \log^{\frac{1}{2}}(1+\alpha T)\lesssim 1.$$
	Note that $\log(1+\alpha T)$ is negligible with respect to $\alpha^{\frac{1}{2}}$ as $\alpha$ is sufficiently large, then by \eqref{delta restriction}, the above inequality holds for small $\epsilon$ and $\rho$.
	
	We can use similar method to prove that the remaining terms of Lemmas \ref{prop3 1st lemma}--\ref{Ia and Ib} satisfy \eqref{prf of prop3}.
	Finally we consider the term $I_{251}(h)$ separately, noticing that $T^{\frac{1}{2}-\rho} \leq T$ as $\rho \in (0,\frac{1}{4})$, then we can easily get the desired conclusion by combining \eqref{I252 final} with \eqref{prf of prop3}.

	\bigskip
	
	\noindent\textbf{Acknowledgement.} The second author is grateful to Umberto Pappalettera for useful discussions on the proof of Proposition \ref{main proposition}. He thanks also the financial supports of the National Key R\&D Program of China (No. 2020YFA0712700), the National Natural Science Foundation of China (Nos. 11931004, 12090010, 12090014), and the Youth Innovation Promotion Association, CAS (Y2021002).

\end{document}